\documentclass[12pt,reqno]{amsart}

\usepackage{amssymb}
\usepackage{hyperref}
\usepackage{bm}
\usepackage{caption,subcaption}
\usepackage{color}
\usepackage{pgf,tikz}
\usepackage{wrap fig}
\usepackage{xypic}
\usepackage{epigraph}

\usepackage{marginnote}
\newcommand{\Gal}{\text{\rm Gal}}
\newcommand{\im}{\text{\rm im}}

\newcommand{\comment}[1]{}

\makeindex

\begin{document}


\thanks{The second author is partially supported by the Natural Sciences and Engineering Research Council of Canada grant R0370A01.  He also gratefully acknowledges the Faculty of Science Distinguished Research Professorship, Western Science, in years 2004/2005 and 2020/2021. The third author is partially supported by 2017--2019 Wellesley College Faculty Awards. The fourth author was supported  in part by National Security Agency grant MDA904-02-1-0061.}

\title{Galois module structure of square power classes for biquadratic extensions}

\author[Frank Chemotti]{Frank Chemotti}
\email{fchemotti@gmail.com}

\author[J\'{a}n Min\'{a}\v{c}]{J\'{a}n Min\'{a}\v{c}}
\address{Department of Mathematics, Western University, London, Ontario, Canada N6A 5B7}
\email{minac@uwo.ca}

\author[Andrew Schultz]{Andrew Schultz}
\address{Department of Mathematics, Wellesley College, 106 Central Street, Wellesley, MA \ 02481 \ USA}
\email{andrew.c.schultz@gmail.com}

\author[John Swallow]{John Swallow}
\address{Office of the President, 2001 Alford Park Drive, Kenosha, WI \ 53140 \ USA}
\email{jswallow@carthage.edu}

\date{\today}

\begin{abstract}
For a Galois extension $K/F$ with $\text{char}(K)\neq 2$ and $\Gal(K/F) \simeq \mathbb{Z}/2\mathbb{Z}\oplus\mathbb{Z}/2\mathbb{Z}$, we determine the $\mathbb{F}_2[\Gal(K/F)]$-module structure of $K^\times/K^{\times 2}$. Although there are an infinite number of (pairwise non-isomorphic) indecomposable $\mathbb{F}_2[\mathbb{Z}/2\mathbb{Z}\oplus\mathbb{Z}/2\mathbb{Z}]$-modules, our decomposition includes at most $9$ indecomposable types.   This paper marks the first time that the Galois module structure of power classes of a field has been fully determined when the modular representation theory allows for an infinite number of indecomposable types.  
\end{abstract}

\maketitle


\newtheorem*{theorem*}{Theorem}
\newtheorem*{lemma*}{Lemma}
\newtheorem{theorem}{Theorem}
\newtheorem{proposition}{Proposition}[section]
\newtheorem{corollary}[proposition]{Corollary}
\newtheorem{lemma}[proposition]{Lemma}

\theoremstyle{definition}
\newtheorem*{definition*}{Definition}
\newtheorem*{remark*}{Remark}
\newtheorem{example}[proposition]{Example}

\parskip=6pt plus 2pt minus 2pt


\section{Introduction}

\subsection{Background and Motivation}  

Let $K$ be a field, and write $\xi_p$ for a primitive $p$th root of unity.  We write $K_{\text{\tiny{sep}}}$ for a separable closure of $K$, and $K(p)$ for the maximal $p$-extension within $K_{\text{\tiny{sep}}}$.  Each of these extensions is Galois.  The absolute Galois group of $K$ is the group $G_K:=\Gal(K_{\text{\tiny{sep}}}/K)$.  The group $G_K(p):=\Gal(K(p)/K)$ is the maximal pro-$p$ quotient of $G_K$.  For convenience, we will call $G_K(p)$ the \emph{absolute $p$-Galois group} of $K$.   One of the major open problems in Galois theory is to determine those profinite groups $G$ for which there exists some field $K$ with $G_K \simeq G$; i.e., to distinguish absolute Galois groups within the class of profinite groups.  This problem is very difficult.  The analogous question for pro-$p$ groups --- to distinguish absolute $p$-Galois groups within the class of pro-$p$ groups --- is also unsolved and extremely difficult.  

How does one look for those properties that distinguish absolute $p$-Galois groups from the broader class of pro-$p$ groups?  To motivate the perspective pursued in this paper, note that since $G_K(p)$ is a pro-$p$ group, it is natural to study it recursively through its Frattini subgroup and its quotient.  This quotient is the maximal elementary $p$-abelian quotient of $G_K(p)$, which by Kummer theory (assuming $\xi_p \in K$) corresponds to $J(K):=K^\times/K^{\times p}$.  In the case that $K$ is itself a Galois extension of a field $F$, one then has a natural action of $\text{Gal}(K/F)$ on $J(K)$.  (Throughout the remainder of this discussion we will assume that $\Gal(K/F)$ is a $p$-group, just to stay firmly planted in the context of $p$-groups.)  One field-theoretic lens for studying $G_K(p)$, therefore, is to determine the structure of $J(K)$ as a module over $\mathbb{F}_p[\Gal(K/F)]$.  It is worth noting that the submodules of $J(K)$ are in bijection with the elementary $p$-abelian extensions of $K$ that are additionally Galois over $F$ (see \cite{Waterhouse}), again assuming $\xi_p \in K$.

Given that the modular representation theory of $\mathbb{F}_p[\Gal(K/F)]$ is most tractable when $\Gal(K/F)$ is cyclic, this is a natural place to begin. Some early work by Borevi\v{c} and Fadeev (see \cite{Bo,F}) examined the module structure of $J(K)$ when $K$ is a local field and $\Gal(K/F) \simeq \mathbb{Z}/p\mathbb{Z}$ using local class field theory.  Subsequently, Mina\v{c} and Swallow (\cite{MS}) 
showed that the module $J(K)$ can be computed when $\Gal(K/F) \simeq \mathbb{Z}/p\mathbb{Z}$ assuming only $\xi_p \in K$ and without such heavy machinery.  

The surprise from this result is twofold.  First, despite the fact that the field $K$ is completely general, the $\mathbb{F}_p[\Gal(K/F)]$-module $J(K)$ is far more stratified than a ``random" $\mathbb{F}_p[\mathbb{Z}/p\mathbb{Z}]$-module: whereas a general $\mathbb{F}_p[\mathbb{Z}/p\mathbb{Z}]$-module can have summands drawn from any one of $p$ possible isomorphism types, the decomposition of $J(K)$ as an $\mathbb{F}_p[\Gal(K/F)]$-module involves at most $3$ isomorphism classes of indecomposable summands (free cyclic modules, trivial cyclic modules, and at most one cyclic module of dimension $2$). The second surprise comes from the proof of the result itself.  Though this decomposition requires a lot of careful work, the machinery needed for the proof is actually quite elementary. Indeed, the key theoretical tool in the proof is Hilbert's Satz 90.  

The benefit of an elementary approach to the decomposition of $J(K)$ when $\Gal(K/F) \simeq \mathbb{Z}/p\mathbb{Z}$ and $\xi_p \in K$ isn't just that it lets one compute this module for arbitrary $K$, but also that it provides a roadmap for how one might generalize this decomposition to a broader class of Galois modules.  Indeed, generalizations of this type have been carried out in a variety of contexts.  In \cite{MSS1}, three of the authors gave the decomposition of $J(K)$ whenever $\Gal(K/F) \simeq \mathbb{Z}/p^n\mathbb{Z}$.  Looking past power classes, observe that when $i = 1$ and $\xi_p \in K$ we have $H^i(G_{K}(p),\mathbb{F}_p) \simeq K^\times/K^{\times p}$ as Galois modules, so these cohomology groups provide a new family of Galois modules to investigate.  Using the connection between Milnor $K$-theory and Galois cohomology --- together with the generalization of Hilbert 90 to this context --- two of the authors and N.~Lemire gave a decomposition of the Galois cohomology groups $H^i(G_{K}(p),\mathbb{F}_p)$ in \cite{LMS} under the assumption that $\Gal(K/F) \simeq \mathbb{Z}/p\mathbb{Z}$ and $\xi_p \in K$.  Some partial results for the structure of $H^i(G_K(p),\mathbb{F}_p)$ when $\Gal(K/F) \simeq \mathbb{Z}/p^n\mathbb{Z}$ are given in \cite{LMSS}.   Generalizations to the case where $K$ is characteristic $p$ (but $\Gal(K/F)$ is still assumed to be a cyclic $p$-group) have also been explored in \cite{BS,BLMS,MSS.K.in.char.p,Schultz}.

As with the original decomposition of $J(K)$ in \cite{MS}, these subsequent module decompositions contain far fewer isomorphism classes of indecomposables than one might expect \emph{a priori}.  These stratified decompositions have, in turn, been translated into properties that distinguish absolute $p$-Galois groups within the larger class of pro-$p$ groups.  For example, using the structure of $J(K)$, a variety of automatic realization and realization multiplicity results have been proved (see \cite{BS,CMSHp3,MSS.auto,MS2,Schultz}).  The module structure for cohomology groups computed in \cite{LMS} was used in \cite{BLMS.prop.groups} to find a number of pro-$p$ groups that are not absolute $p$-Galois groups.  

It would be natural to assume that the previous module computations are possible because the modular representation theory for the group ring $\mathbb{F}_p[\Gal(K/F)]$ is simple when $\Gal(K/F)$ is cyclic --- namely, in this case there are $|\Gal(K/F)|$  isomorphism classes of indecomposables, and each of them is cyclic. In contrast, if $G$ is a non-cyclic elementary $p$-abelian group then there are infinitely many isomorphism classes of indecomposable $\mathbb{F}_p[G]$-modules (and often it is impossible to give a full classification of indecomposables). There has been some work which provides partial information about Galois modules in these more complicated settings, recovering information about the Socle series or arguing that the modules are constant Jordan type in  special cases (\cite{AGKM,Eimer,MST}). However, these modules were not determined completely.   

In this paper we provide a decomposition for $J(K)$ as an $\mathbb{F}_2[\Gal(K/F)]$-module when $\Gal(K/F) \simeq \mathbb{Z}/2\mathbb{Z} \oplus \mathbb{Z}/2\mathbb{Z}$, without any restriction on $K$ other than $\text{char}(K) \neq 2$.  The decomposition follows the two themes that have arisen in the context of cyclic Galois groups: the module structure is far more stratified than one would expect for a general module (across all fields $K$, the summands are drawn from at most $9$ indecomposable types), and the decomposition can be determined using relatively concrete techniques and the assistance of Hilbert 90 (see \cite{DMSS} for a discussion on how one interprets Hilbert 90 for biquadratic extensions). Undoubtedly this stratified decomposition --- both the appearance of some summand types and the exclusion of others --- can be translated into new and exciting group-theoretic properties of absolute $2$-Galois groups. The authors are currently looking into such results.

A decomposition of $J(K)$ when $\Gal(K/F) \simeq \mathbb{Z}/2\mathbb{Z} \oplus \mathbb{Z}/2\mathbb{Z}$ was completed by the first, second, and fourth authors in 2005 using more technical machinery. A deeper dive into the module from this perspective was explored in \cite{Ferguson} under 
the joint supervision of Mina\v{c} and Swallow.   The impetus for revisiting this problem using more ubiquitous tools was to give greater insight into how decompositions for $J(K)$ (and its ilk) could be carried out when $\Gal(K/F)$ is some other elementary $p$-abelian group. This approach has already met with success: it has allowed us to exclude one summand type that appeared in the original decomposition from 2005, and it inspired the recent decomposition of the parameterizing space of elementary $p$-abelian extensions of $K$ as a module over $\mathbb{F}_p[\Gal(K/F)]$ whenever $G_F(p)$ is a free pro-$p$ group and $\Gal(K/F)$ is \emph{any} finite $p$-group (see the remark after Theorem \ref{th:main.theorem}).  We are hopeful that the techniques we develop here can inspire the next steps towards investigations of a broader class of elementary $p$-abelian Galois modules.  

\subsection{Statement of Main Result}

We first set terminology that will hold for the balance of the paper.  Suppose that $F$ is a field with $\text{char}(F) \neq 2$ and that $K/F$ is an extension with $G:=\Gal(K/F) \simeq \mathbb{Z}/2\mathbb{Z}\oplus\mathbb{Z}/2\mathbb{Z}$.  Let $a_1,a_2 \in F$ be given so that $K = F(\sqrt{a_1},\sqrt{a_2})$, and let $\sigma_1,\sigma_2 \in \text{Gal}(K/F)$ be their duals; that is, we have $\sigma_i(\sqrt{a_j}) = (-1)^{\delta_{ij}}\sqrt{a_j}$.  For $i \in \{1,2\}$ we define $K_i = F(\sqrt{a_i})$.  Write $H_i$ for the subgroup of $\Gal(K/F)$ which fixes elements in $K_i$, and $\overline{G_i}$ for the corresponding quotient group: $\overline{G_i} := \Gal(K_i/F) = \{\overline{\text{id}},\overline{\sigma_i}\}.$  In the same spirit, write $K_3 = F(\sqrt{a_1a_2})$, denote the subgroup of $\Gal(K/F)$ which fixes $K_3$ as $H_3$, and use $\overline{G_3}$ for the corresponding quotient $G/H_3 = \Gal(K_3/F)$.  To round out the notation, let $H_0 = \{\text{id}\}$ (the elements which fix the extension $K/F$) and $H_4 = \Gal(K/F)$ (the elements which fix the extension $F/F$), and use $\overline{G_0}$ and $\overline{G_4}$ for their quotients. (See Figure \ref{fig:lattice.of.fields}.)

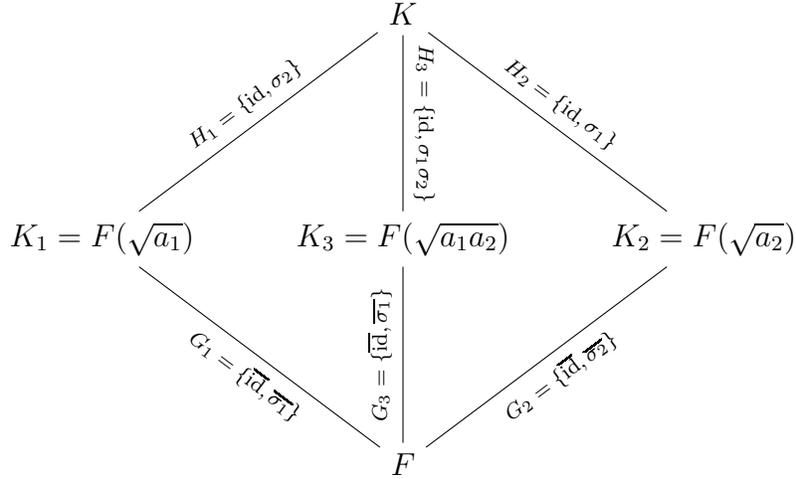
\begin{figure}[h]
        \begin{tikzpicture}[scale=1]
		\node (K) at (0,0) {$K$};
		\node (F) at (0,-6) {$F$};
		\node (K1) at (-4,-3) {$K_1 = F(\sqrt{a_1})$};
		\node (K2) at (4,-3) {$K_2 = F(\sqrt{a_2})$};
		\node (K3) at (0,-3) {$K_3=F(\sqrt{a_1a_2})$};
		\draw[-] (K)-- node[sloped,above] {\tiny $H_1 = \{\text{id},\sigma_2\}$} (K1);
		\draw[-] (K)-- node[sloped,above] {\tiny $H_2=\{\text{id},\sigma_1\}$} (K2);
		\draw[-] (K)-- node[sloped,above] {\tiny $H_3=\{\text{id},\sigma_1\sigma_2\}$} (K3);
		\draw[-] (F)-- node[sloped,below] {\tiny $G_1 = \{\overline{\text{id}},\overline{\sigma_1}\}$} (K1);
		\draw[-] (F)-- node[sloped,below] {\tiny $G_2 = \{\overline{\text{id}},\overline{\sigma_2}\}$} (K2);
		\draw[-] (F)-- node[sloped,above] {\tiny $G_3 = \{\overline{\text{id}},\overline{\sigma_1}\}$} (K3);
        \end{tikzpicture}
        \caption{The lattice of fields for $K/F$, with corresponding Galois groups}\label{fig:lattice.of.fields}
\end{figure}



In our result below, we use $\Omega^{-n}$ and $\Omega^n$ to denote certain indecomposable modules of dimension $2n+1$; more information on these modules can be found in Section \ref{sec:module.stuff}.  

\begin{theorem}\label{th:main.theorem}
\suppressfloats[t]
Suppose that $\text{char}(K) \neq 2$ and that $\Gal(K/F) \simeq \mathbb{Z}/2\mathbb{Z}\oplus\mathbb{Z}/2\mathbb{Z}$.  Let $J(K) = K^\times/K^{\times 2}$.  Then as an $\mathbb{F}_2[\Gal(K/F)]$-module we have
$$J(K) \simeq X \oplus Y_0 \oplus Y_1 \oplus Y_2 \oplus Y_3 \oplus Y_4 \oplus Z_1 \oplus Z_2,$$ where
\begin{itemize}
\item $X$ is isomorphic to one of the following: $\{0\}, \mathbb{F}_2,\mathbb{F}_2\oplus\mathbb{F}_2, \Omega^{-1}, \Omega^{-2},$ or $\Omega^{-1}\oplus\Omega^{-1}$;
\item for each $i \in \{0,1,2,3,4\}$, the summand $Y_i$ is a direct sum of modules isomorphic to $\mathbb{F}_2[\overline{G_i}]$; and
\item for each $i \in \{1,2\}$, the summand $Z_i$ is a direct sum of modules isomorphic to $\Omega^i$.
\end{itemize}
\end{theorem}

\begin{remark*}When $\text{char}(K) = 2$, elementary $2$-abelian extensions of $K$ are parameterized by $\mathbb{F}_2$-subspaces of $K/\wp{K}$, where $\wp{K} = \{k^2-k: k \in K\}$.  It is therefore natural to ask whether $K/\wp(K)$ can be decomposed as an $\mathbb{F}_2[\Gal(K/F)]$-module as well.  The answer is a resounding ``yes." In fact, this is a special case of a far more general theorem from the forthcoming paper \cite{HMNST}: for any prime $p$ and any Galois extension $K/F$ so that $\Gal(K/F)$ is a finite $p$-group, if $G_F(p)$ is a free pro-$p$ group, then the parameterizing space of elementary $p$-abelian extensions of $K$ decomposes into a free summand and a single summand isomorphic to $\Omega^{-2}_{\Gal(K/F)}$. 
\end{remark*}

\subsection{Outline of paper}

In Section \ref{sec:module.stuff} we review some basic facts concerning modules over $\mathbb{F}_2[\mathbb{Z}/2\mathbb{Z}\oplus\mathbb{Z}/2\mathbb{Z}]$
.  Section \ref{sec:F.spanning.candidate} is devoted to producing a ``large" module whose fixed part is the ``obvious" componenent $[F^\times]$ within $J(K)^G$; the key is to give a filtration of $[F^\times]$ that is sensitive to image subspaces coming from particular elements of $\mathbb{F}_2[G]$. Section \ref{sec:X.candidate} aims to find a module whose fixed part spans a complement to $[F^\times]$ in $J(K)^G$.  This requires a deeper understanding of how $J(K)^G$ behaves under the norm maps associated to the intermediate extensions $K/K_i$ (for $i \in \{1,2,3\}$). The proof of Lemma \ref{le:compatible.supernorms.means.preimage.exists.somewhere} gives our first appearance of a Hilbert 90 result for biquadratic extensions.  Section \ref{sec:proof.of.main.result} has another result related to Hilbert 90 for biquadratic extensions (Lemma \ref{le:chain.breaking.lemma}), as well as the proof of Theorem \ref{th:main.theorem}. In Section \ref{sec:realizability} we discuss the realizability of some of the possibilities for the $X$ summand in terms of the solvability (or non-solvability) of particular embedding problems.

\subsection{Acknowledgements}

We gratefully acknowledge discussions and collaborations with our friends and colleagues D.~Benson, B.~Brubaker, J.~Carlson, S.~Chebolu, I.~Efrat, J.~G\"{a}rtner, S.~Gille, L.~Heller, D.~Hoffmann, J.~Labute, T.-Y.~Lam, R.~Sharifi, N.D.~Tan, A.~Topaz, R.~Vakil, K.~Wickelgren, O.~Wittenberg, which have influenced our work in this and related papers. We are particularly grateful to A.~Eimer and P.~Guillot for their careful consideration of a previous draft of this manuscript which omitted $\Omega^2$ summands.

\section{A Primer in Diagramatic Thinking in Module Theory}\label{sec:module.stuff}

We will use $G$ to denote the Klein $4$-group with generators $\sigma_1$ and $\sigma_2$.  When $M$ is an $\mathbb{F}_2[G]$-module, we assume that $M$'s structure is multiplicative, so that the module action is written exponentially.  Despite this, if $U,V$ are submodules of a larger $\mathbb{F}_2[G]$-module $W$, we will still write $U+V$ for the set $\{uv: u \in U, v \in V\}$, and we will use $U \oplus V$ to indicate this set when $U \cap V$ is trivial.

Throughout this paper we will be considering the solvability of certain systems of equations within various $\mathbb{F}_2[G]$-modules.  Though one could of course write these systems out, it will often be convenient to have graphical representations for the equations.  We adopt the convention that an arrow between elements denotes that one is the image of another through some given element of $F_2[G]$, with the direction of the arrow indicating the acting element from $\mathbb{F}_2[G]$.  If the arrow points down and to the left, this indicates that the bottom element is the image of the upper element under $1+\sigma_2$, and likewise if the arrow points down and to the right this means the lower element is the image of the upper element under $1+\sigma_1$.  In the event that the action of $1+\sigma_1$ and $1+\sigma_2$ is the same on a given element, then we write the image immediately below, and use two bent arrows to signify the equality of the two actions.   Figure \ref{fig:cyclic.module.depictions} gives some basic examples.

\begin{figure}[!ht]
\centering
\begin{subfigure}[!ht]{.3\textwidth}
\centering
\begin{tikzpicture}[scale=.75]
\node (M1) at (0,0) {$\alpha$};
\node (M2) at (-2,-2) {$\alpha_1$};
\draw[->] (M1) --  node[sloped,above] {\tiny $1+\sigma_2$}(M2);
\end{tikzpicture} 
\end{subfigure}
\begin{subfigure}[h]{.3\textwidth}
  \centering
\begin{tikzpicture}[scale=.75]
\node (M1) at (0,0) {$\beta$};
\node (M2) at (0,-2) {$\beta_1$};
\draw[->, out=225, in=135] (M1) to node[sloped,above] {\tiny $1+\sigma_2$}(M2);
\draw[->, out=315, in=45] (M1) to node[sloped,above] {\tiny $1+\sigma_1$}(M2);
\end{tikzpicture}
\end{subfigure}
\begin{subfigure}[h]{.3\textwidth}
  \centering
\begin{tikzpicture}[scale=.75]
\node (M1) at (0,0) {$\gamma$};
\node (M2) at (-2,-2) {$\gamma_1$};
\node (M3) at (2,-2) {$\gamma_2$};
\draw[->] (M1) to node[sloped,above] {\tiny $1+\sigma_2$} (M2);
\draw[->] (M1) -- node[sloped,above] {\tiny $1+\sigma_1$}(M3);
\end{tikzpicture}
\end{subfigure}
\caption{A sampling of linear equations.  On the left we have the relation $\alpha^{1+\sigma_2} = \alpha_1$; in the middle we have the simultaneous equations in  $\beta$ and $\beta_1$ given by $\beta^{1+\sigma_1} = \beta^{1+\sigma_2} = \beta_1$; and on the right we have the simultaneous equations in $\gamma,\gamma_1,\gamma_2$ given by $\gamma^{1+\sigma_2} = \gamma_1$ and $\gamma^{1+\sigma_1} = \gamma_2$.}
\label{fig:cyclic.module.depictions}
\end{figure}
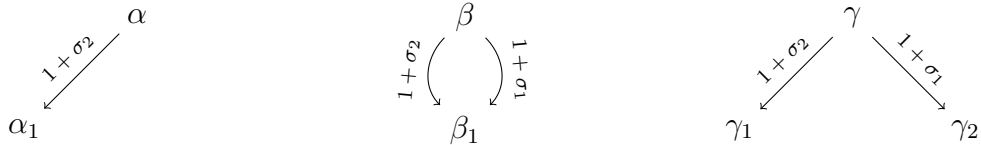

Since these diagrams represent simultaneous linear equations in the module, we will say that a solution to a system of equations is a solution to the corresponding diagram; if we have some fixed values for particular parameters in a system of equations, and there exist values for the remaining parameters so that the underlying system is solved, then we will say that the diagram is solvable for those (original) fixed values.  For example, to say that the diagram on the left side of Figure \ref{fig:cyclic.module.depictions} is solvable for some particular $\alpha_1$ is equivalent to saying that $\alpha_1$ is in the image of $1+\sigma_1$.

Our decomposition will not require us to have a classification of indecomposable $\mathbb{F}_2[G]$-modules, but for the reader's benefit we review some basic information about these modules. For a full treatment, the reader can consult \cite[Theorem 4.3.3]{Benson}.  There are seven ideals in the ring $\mathbb{F}_2[\mathbb{Z}/2\mathbb{Z}\oplus\mathbb{Z}/2\mathbb{Z}]$, and hence six cyclic, non-trivial indecomposable submodule classes.  Aside from the even-dimensional cyclic modules, there are also families of indecomposable even-dimensional $\mathbb{F}_2[G]$-modules that correspond to certain rational canonical form matrices.  These will not appear in our decomposition.  There are also odd-dimensional indecomposable $\mathbb{F}_2[G]$-modules: for each odd number $2n+1$ with $n>1$, there are two irreducible $\mathbb{F}_2[G]$-modules of dimension $n$, denoted $\Omega^n$ and $\Omega^{-n}$.  As it happens, our decomposition of $J(K)$ will only require the cyclic modules we have already introduced together with $\Omega^{1}, \Omega^{2}, \Omega^{-1}$, and $\Omega^{-2}$.  We will need formal definitions for these latter modules, but there is no additional cost to define $\Omega^{n}$ and $\Omega^{-n}$ in general.  Using our depiction scheme, these modules are shown in Figure \ref{fig:noncyclic.module.types}. 
\begin{figure}[!h]
\begin{tikzpicture}[scale=.65, every node/.style={scale=1}]
\node [font=\fontsize{17pt}{17pt}\selectfont] at (-6,-1) {$\Omega^{-n}:$};
\node (A1) at (0,0) {$\alpha_1$};
\node (A2) at (4,0) {$\alpha_2$};
\node (A3) at (8,0) {\ };
\node at (9,-1) {$\cdots$};
\node (An1) at (10,0) {\ };
\node (An) at (14,0) {$\alpha_n$};
\node (B1) at (-2,-2) {$\beta_1$};
\node (B2) at (2,-2) {$\beta_2$};
\node (B3) at (6,-2) {$\beta_3$};
\node (Bn) at (12,-2) {$\beta_n$};
\node (Bn1) at (16,-2) {$\beta_{n+1}$};
\draw[->] (A1) to node[sloped,above] {\tiny $1+\sigma_2$}(B1);
\draw[->] (A1) to node[sloped,above] {\tiny $1+\sigma_1$}(B2);
\draw[->] (A2) to node[sloped,above] {\tiny $1+\sigma_2$}(B2);
\draw[->] (A2) to node[sloped,above] {\tiny $1+\sigma_1$}(B3);
\draw[->] (A3) to node[sloped,above] {\tiny $1+\sigma_2$}(B3);
\draw[->] (An1) to node[sloped,above] {\tiny $1+\sigma_1$} (Bn);
\draw[->] (An) to node[sloped,above] {\tiny $1+\sigma_2$}(Bn);
\draw[->] (An) to node[sloped,above] {\tiny $1+\sigma_1$} (Bn1);

\node [font=\fontsize{17pt}{17pt}\selectfont] at (-6,-5) {$\Omega^{n}:$};
\node (C1) at (-2,-4) {$\gamma_1$};
\node (C2) at (2,-4) {$\gamma_2$};
\node (C3) at (6,-4) {$\gamma_3$};
\node at (9,-5) {$\cdots$};
\node (Cn) at (12,-4) {$\gamma_n$};
\node (Cn1) at (16,-4) {$\gamma_{n+1}$};
\node (D0) at (-4,-6) {$1$};
\node (D1) at (0,-6) {$\delta_1$};
\node (D2) at (4,-6) {$\delta_2$};
\node (D3) at (8,-6) {\ };
\node (Dnm1) at (10,-6) {\ };
\node (Dn) at (14,-6) {$\delta_{n}$};
\node (Dnp1) at (18,-6) {$1$};
\draw[->] (C1) to node[sloped,above] {\tiny $1+\sigma_2$}(D0);
\draw[->] (C1) to node[sloped,above] {\tiny $1+\sigma_1$}(D1);
\draw[->] (C2) to node[sloped,above] {\tiny $1+\sigma_2$}(D1);
\draw[->] (C2) to node[sloped,above] {\tiny $1+\sigma_1$}(D2);
\draw[->] (C3) to node[sloped,above] {\tiny $1+\sigma_2$}(D2);
\draw[->] (C3) to node[sloped,above] {\tiny $1+\sigma_1$} (D3);
\draw[->] (Cn) to node[sloped,above] {\tiny $1+\sigma_2$}(Dnm1);
\draw[->] (Cn) to node[sloped,above] {\tiny $1+\sigma_1$} (Dn);
\draw[->] (Cn1) to node[sloped,above] {\tiny $1+\sigma_2$}(Dn);
\draw[->] (Cn1) to node[sloped,above] {\tiny $1+\sigma_1$}(Dnp1);
\end{tikzpicture}
\caption{The two indecomposable $\mathbb{F}_2[G]$-modules of odd dimension $2n+1$}
\label{fig:noncyclic.module.types}
\end{figure}
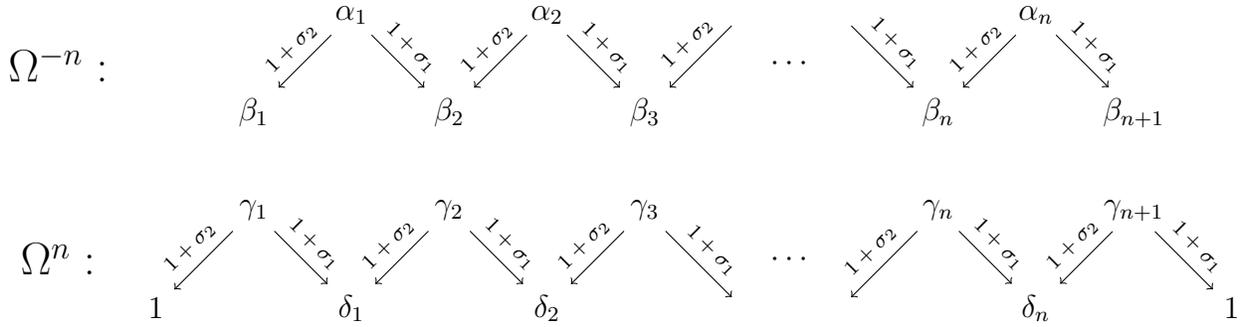


One key fact we'll use about $\mathbb{F}_2[G]$-module is that we can detect independence of two $\mathbb{F}_2[G]$-modules by examining the independence of their fixed parts.  We follow the standard convention of writing $M^G$ for the fixed submodule of an $\mathbb{F}_2[G]$-module $M$.  
\begin{lemma}\label{le:exclusion.lemma}
Suppose that $M$ and $N$ are submodules of a larger $\mathbb{F}_2[G]$-module $W$.  Then $M \cap N = \{1\}$ if and only if $M^G \cap N^G = \{1\}$.
\end{lemma}
\begin{proof}
Of course if $M \cap N = \{1\}$, then $M^G \cap N^G = \{1\}$ as well.  Suppose, then, that $M^G \cap N^G = \{1\}$, and let $w \in M \cap N$ be given.  If $w$ is nontrivial, then $\langle w \rangle$ is isomorphic to precisely one of the following: $\mathbb{F}_2$, $\mathbb{F}_2[\overline{G_1}]$, $\mathbb{F}_2[\overline{G_2}]$, $\mathbb{F}_2[\overline{G_3}]$, $\Omega^{-1}$, or $\mathbb{F}_2[G]$.  In the first case we have $w \in W^G$, and so $w \in M^G \cap N^G = \{1\}$; this is a contraction.  If either $\langle w \rangle \simeq \mathbb{F}_2[\overline{G_1}]$, $\langle w\rangle \simeq \mathbb{F}_2[\overline{G_3}]$, or $\langle w\rangle \simeq  \Omega^{-1}$, then $w^{1+\sigma_1}$ is a nontrivial element in $W^G$; but this again leads to a contradiction, since then we again have $w^{1+\sigma_1} \in M^G \cap N^G = \{1\}$.  If $\langle w \rangle \simeq \mathbb{F}_2[\overline{G_2}]$, then $w^{1+\sigma_2}$ is the nontrivial element in $W^G$ which leads to a contradiction, and if $\langle w \rangle \simeq \mathbb{F}_2[G]$ then the contradictory nontrivial element of $W^G$ is $w^{(1+\sigma_1)(1+\sigma_2)}$.
\end{proof}

\section{A maximal submodule with fixed part $[F^\times]$}\label{sec:F.spanning.candidate}

In Section \ref{sec:module.stuff} we saw that fixed submodules play an important role in determining independence amongst $\mathbb{F}_2[G]$-modules.  Of course the most natural fixed submodule of $J(K)$ is $[F^\times]$.  Our objective in this section will be to find a ``sufficiently large" submodule $\hat J$ of $J(K)$ for which $\hat J^G = [F^\times]$.  For the purposes of the decomposition that we are building, being ``sufficiently large" will mean that $\hat J$ contains solutions to certain systems of equations, assuming such equations have solutions within the full module $J(K)$.  

In a certain sense, we are most interested in finding free summands --- by which we mean free over $\mathbb{F}_2[\overline{G_i}]$ for some $i \in \{0,1,2,3,4\}$) --- with the general philosophy that larger submodules are preferrable.  Hence primary preference goes to free (cyclic) $\mathbb{F}_2[G]$-modules, and secondary preference goes to free (cyclic) $\mathbb{F}_2[\overline{G_i}]$-modules for $i \in \{1,2,3\}$; for concreteness, we give preference to $i=1$ over $i=2$, and $i=2$ over $i=3$.  We finish with free $\mathbb{F}_2[\overline{G_4}]$-modules (i.e., trivial modules).  

The issue in pursuing this agenda is that there are potential interrelations between these free modules.  For example, suppose a free cyclic $\mathbb{F}_2[\overline{G_1}]$-module $\langle [\gamma_1]\rangle$ and a free cyclic $\mathbb{F}_2[\overline{G_2}]$-module $\langle [\gamma_2]\rangle$ share the same fixed submodule $\langle [f]\rangle$.  This means that $[f],[\gamma_1],$ and $[\gamma_2]$ satisfy the system of equations 
$$
\begin{tikzpicture}[scale=0.65]
\node (B1) at (10,0) {$[\gamma_1]$};
\node (B2) at (14,0) {$[\gamma_2]$};
\node (B0) at (12,-2) {$[f]$};
\node (Bt1) at (8,-2) {$[1]$};
\node (Bt2) at (16,-2) {$[1].$};

\draw[->] (B1) to (B0);
\draw[->] (B2) to (B0);
\draw[->] (B1) to (Bt1);
\draw[->] (B2) to (Bt2);
\end{tikzpicture}
$$
Hence in our pursuit of free submodules, we are obliged to look for solutions to this type of system and ensure our decomposition of $[F^\times]$ captures these elements.  

With all this in mind, let us move toward statements that are more precise.  In Figure \ref{fig:first.filtration.diagrams} we introduce 5 subspaces of $[F^\times]$ that capture the ideas we alluded to in the previous paragraphs.  We denote these spaces $\mathfrak{A},\mathfrak{V},\mathfrak{B},\mathfrak{C},$ and $\mathfrak{D}$.  For $\mathfrak{M} \in \{\mathfrak{A},\mathfrak{V},\mathfrak{B},\mathfrak{C},\mathfrak{D}\}$, the space $\mathfrak{M}$ is the set of all $[f] \in [F^\times]$ for which the corresponding diagram from Figure \ref{fig:first.filtration.diagrams} is solvable for $[f]$.  For example, an element $[f] \in [F^\times]$ is an element of $\mathfrak{A}$ if and only if $[f] \in [N_{K/F}(K^\times)]$ (since $N_{K/F}$ is given by applying $(1+\sigma_1)(1+\sigma_2)$).

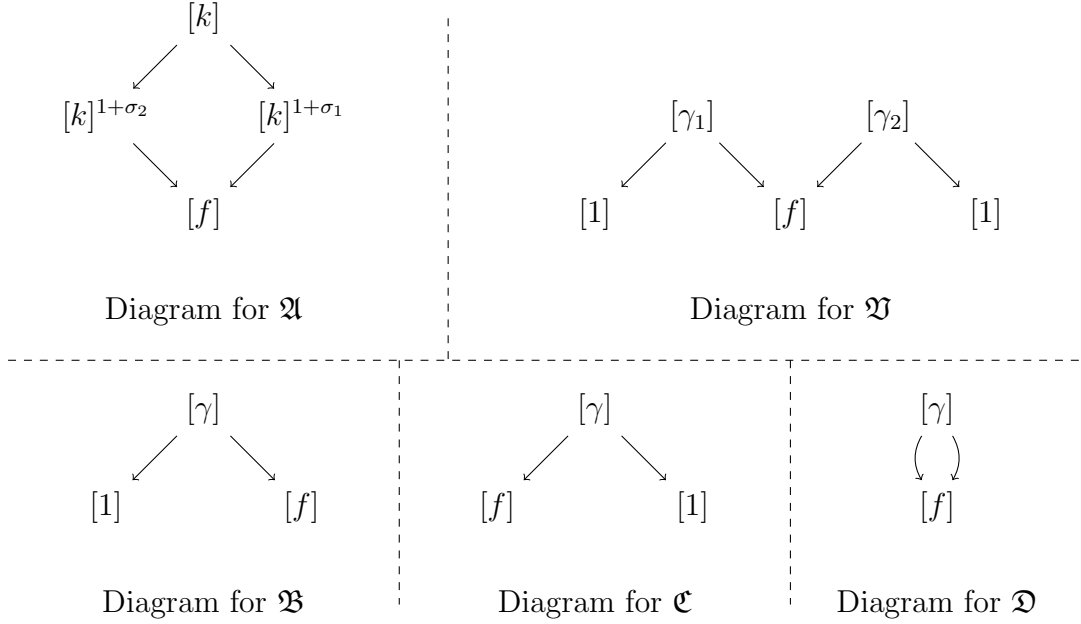
\begin{figure}[!ht]
\begin{tikzpicture}[scale=.65]
\node (A3) at (0,2) {$[k]$};
\node (A1) at (-2,0) {$[k]^{1+\sigma_2}$};
\node (A2) at (2,0) {$[k]^{1+\sigma_1}$};
\node (A0) at (0,-2) {$[f]$};
\node at (0,-4) {Diagram for $\mathfrak{A}$};

\draw[->] (A3) to (A1);
\draw[->] (A3) to (A2);
\draw[->] (A1) to (A0);
\draw[->] (A2) to (A0);

\draw[dashed] (5,2) -- (5,-5);

\node (B1) at (10,0) {$[\gamma_1]$};
\node (B2) at (14,0) {$[\gamma_2]$};
\node (B0) at (12,-2) {$[f]$};
\node (Bt1) at (8,-2) {$[1]$};
\node (Bt2) at (16,-2) {$[1]$};
\node at (12,-4) {Diagram for $\mathfrak{V}$};

\draw[->] (B1) to (B0);
\draw[->] (B2) to (B0);
\draw[->] (B1) to (Bt1);
\draw[->] (B2) to (Bt2);

\draw[dashed] (-4,-5) -- (18,-5);

\node (C1) at (0,-6) {$[\gamma]$};
\node (Ct) at (-2,-8) {$[1]$};
\node (C2) at (2,-8) {$[f]$};
\node at (0,-10) {Diagram for $\mathfrak{B}$};

\draw[->] (C1) to (C2);
\draw[->] (C1) to (Ct);

\draw[dashed] (4,-5) -- (4,-10);

\node (D1) at (8,-6) {$[\gamma]$};
\node (D2) at (6,-8) {$[f]$};
\node (Dt) at (10,-8) {$[1]$};
\node at (8,-10) {Diagram for $\mathfrak{C}$};

\draw[->] (D1) to (D2);
\draw[->] (D1) to (Dt);

\draw[dashed] (12,-5) -- (12,-10);

\node (E1) at (15,-6) {$[\gamma]$};
\node (E2) at (15,-8) {$[f]$};
\node at (15,-10) {Diagram for $\mathfrak{D}$};

\draw[->, out=240, in=120] (E1) to (E2);
\draw[->, out=300, in=60] (E1) to (E2);
\end{tikzpicture}
\caption{Diagrams that represent the various systems of equations that are solvable in order for $[f]$ to be an element of the subspaces $\mathfrak{A},\mathfrak{V},\mathfrak{B},\mathfrak{C}$ or $\mathfrak{D}$.}
\label{fig:first.filtration.diagrams}
\end{figure}

It is readily apparent that $\mathfrak{A} \subseteq \mathfrak{V}$, and furthermore that $\mathfrak{V}$ is a subspace of both $\mathfrak{B}$ and $\mathfrak{C}$.  We just observed that $\mathfrak{B} \cap \mathfrak{C} = \mathfrak{V}$.  Continuing in the theme of being careful about interrelations that exist between these subspaces, the following lemma considers how elements of $\mathfrak{D}$ are related to elements from $\mathfrak{B}+\mathfrak{C}$.

\begin{lemma}\label{le:characterizing.W.elements}
Let $\mathfrak{B},\mathfrak{C}$, and $\mathfrak{D}$ be defined as in Figure \ref{fig:first.filtration.diagrams}.  Then $[b][c] \in (\mathfrak{B}+\mathfrak{C}) \cap \mathfrak{D}$ if and only if 

\begin{equation}\label{eq:W.diagram}
\begin{tikzpicture}[scale=0.65]
\node (Ut1) at (-4,0) {$[\gamma_1]$};
\node (Ut2) at (0,0) {$[\gamma_2]$};
\node (Ut3) at (4,0) {$[\gamma_3]$};
\node (Ub1) at (-2,-2) {$[b]$};
\node (Ub2) at (2,-2) {$[c]$};
\node (Ub0) at (-6,-2) {$[1]$};
\node (Ub3) at (6,-2) {$[1]$};

\draw[->] (Ut1) to (Ub0);
\draw[->] (Ut1) to (Ub1);
\draw[->] (Ut2) to (Ub1);
\draw[->] (Ut2) to (Ub2);
\draw[->] (Ut3) to (Ub2);
\draw[->] (Ut3) to (Ub3);
\end{tikzpicture}
\end{equation}

is solvable for some $[\gamma_1],[\gamma_2],[\gamma_3] \in J(K)$.
\end{lemma}

\begin{proof}
Suppose first that Equation (\ref{eq:W.diagram}) holds.  From this we see that 
\begin{align*}
[\gamma_1][\gamma_2][\gamma_3]^{1+\sigma_2} &= [1][b][c]\\
[\gamma_1][\gamma_2][\gamma_3]^{1+\sigma_1} &= [b][c][1].
\end{align*}
Hence $[b][c] \in (\mathfrak{B}+\mathfrak{C})\cap \mathfrak{D}$.

For the other direction, suppose there exists $[\gamma] \in J(K)$ so that the diagram for $\mathfrak{D}$ holds with $[\gamma]$ and $[f] = [b][c]$.  Since $[b] \in \mathfrak{B}$ and $[c] \in \mathfrak{C}$, we also have elements $[\gamma_L],[\gamma_R] \in J(K)$ so that $[\gamma_L]$ and $[b]$ satisfy the diagram for $\mathfrak{B}$, and $[\gamma_R]$ and $[c]$ satisfy the diagram for $\mathfrak{C}$.  From these relations we find that Equation \ref{eq:W.diagram} is satisfied with $[\gamma_1] = [\gamma_L]$, $[\gamma_2] = [\gamma_L][\gamma][\gamma_R]$, and $[\gamma_3] = [\gamma_R]$.
\end{proof}

Notice that if $[f]$ and $[\hat f]$ are elements of $\mathfrak{V}$ (with corresponding elements $[\gamma_L],[\gamma_R]$ solving the diagram with $[f]$, and $[\widehat{\gamma_L}],[\widehat{\gamma_R}]$ solving the diagram for $[\hat f]$), then Equation \ref{eq:W.diagram} is solvable for $[b]=[f]$ and $[c]=[\hat f]$:

$$
\begin{tikzpicture}[scale=0.65]
\node (Ut1) at (-4,0) {$[\gamma_L]$};
\node (Ut2) at (0,0) {$[\gamma_R][\widehat{\gamma_L}]$};
\node (Ut3) at (4,0) {$[\widehat{\gamma_R}]$};
\node (Ub1) at (-2,-2) {$[f]$};
\node (Ub2) at (2,-2) {$[\hat f]$};
\node (Ub0) at (-6,-2) {$[1]$};
\node (Ub3) at (6,-2) {$[1].$};

\draw[->] (Ut1) to (Ub0);
\draw[->] (Ut1) to (Ub1);
\draw[->] (Ut2) to (Ub1);
\draw[->] (Ut2) to (Ub2);
\draw[->] (Ut3) to (Ub2);
\draw[->] (Ut3) to (Ub3);
\end{tikzpicture}
$$
Hence we will be particularly interested in understanding solutions to Equation (\ref{eq:W.diagram}) that come from outside $\mathfrak{V}$.  The following lemma characterizes such solutions.

\begin{lemma}\label{le:BW.and.CW}
Suppose that $B$ is a complement to $\mathfrak{V}$ within $\mathfrak{B}$, and that $C$ is a complement to $\mathfrak{V}$ within $\mathfrak{C}$.  Define subspaces $B_W$ of $B$ and $C_W$ of $C$ by 
\begin{align*}
B_W &= \{[b] \in B: \exists [c] \in C \text{ so that Equation (\ref{eq:W.diagram}) is solvable}\}\\
C_W &= \{[c] \in C: \exists [b] \in B \text{ so that Equation (\ref{eq:W.diagram}) is solvable}\}.
\end{align*}
Then there exists an $\mathbb{F}_2$-linear bijection $\phi_W:B_W \to C_W$ which takes each $[b] \in B_W$ to the unique $[c] \in C_W$ for which Equation (\ref{eq:W.diagram}) is solvable for $[b]$ and $[c]$.  
\end{lemma}

\begin{proof}
That $B_W$ and $C_W$ are subspaces follows since Equation (\ref{eq:W.diagram}) is linear.  Now we claim that for each $[b] \in B_W$ there exists a unique $[c] \in C_W$ for which the equations represented by Equation (\ref{eq:W.diagram}) are solvable.  Suppose instead we had some $[b] \in B_W$ so that there exist $[c],[\hat c] \in C_W$ which make Equation (\ref{eq:W.diagram}) solvable; we'll write $[\gamma_1],[\gamma_2],[\gamma_3]$ for the additional terms that solve the system with $[b]$ and $[c]$, and we'll write $[\gamma_1],[\widehat{\gamma_2}],[\widehat{\gamma_3}]$ for the additional terms that solve the system with $[b]$ and $[\hat c]$.  By multiplying the two systems, we're left with 
$$
\begin{tikzpicture}[scale=0.65]
\node (Ut1) at (-4,0) {$[\gamma_1]^2$};
\node (Ut2) at (0,0) {$[\gamma_2][\widehat{\gamma_2}]$};
\node (Ut3) at (4,0) {$[\gamma_3][\widehat{\gamma_3}]$};
\node (Ub1) at (-2,-2) {$[b]^2$};
\node (Ub2) at (2,-2) {$[c][\hat c]$};
\node (Ub0) at (-6,-2) {$[1]$};
\node (Ub3) at (6,-2) {$[1]$};

\draw[->] (Ut1) to (Ub0);
\draw[->] (Ut1) to (Ub1);
\draw[->] (Ut2) to (Ub1);
\draw[->] (Ut2) to (Ub2);
\draw[->] (Ut3) to (Ub2);
\draw[->] (Ut3) to (Ub3);

\end{tikzpicture}
$$
and so we see that $[c][\hat c] \in \mathfrak{V}$.  But since $[c]$ and $[\hat c]$ are contained in a complement $C$ of $\mathfrak{V}$ within $\mathfrak{C}$, this implies $[c][\hat c] =[1]$.  Hence $[c]=[\hat c]$.

The same argument, of course, shows that for a given $[c] \in C_W$ there exists a unique $[b] \in B_W$ for which Equation (\ref{eq:W.diagram}) is solvable.  We define $\phi_W$ as the function which associates to each $[b] \in B_W$ its corresponding $[c] \in C_W$.  The fact that the equations represented by Equation (\ref{eq:W.diagram}) are linear implies that $\phi_W$ is linear as well, and hence an isomorphism of $\mathbb{F}_2$-spaces.
\end{proof}


We are now prepared to state and prove the main result in this section.
\begin{theorem}\label{th:trivial.candidate} 
There exists a submodule $\hat J$ so that $\hat J^G = [F^\times]$, and for which $$\hat J \simeq Y_A \oplus Y_V \oplus Y_W \oplus Y_{B}\oplus Y_{C} \oplus Y_{D} \oplus Y_F$$ where 
\begin{itemize}
\item $Y_A$ is a direct sum of submodules isomorphic to $\mathbb{F}_2[G]$;
\item $Y_V$ is a direct sum of submodules isomorphic to $\Omega^{1}$;
\item $Y_W$ is a direct sum of submodules isomorphic to $\Omega^2$;
\item $Y_{B}$ is a direct sum of submodules isomorphic to $\mathbb{F}_2[\overline{G_1}]$;
\item $Y_{C}$ is a direct sum of submodules isomorphic to $\mathbb{F}_2[\overline{G_2}]$;
\item $Y_{D}$ is a direct sum of submodules isomorphic to $\mathbb{F}_2[\overline{G_3}]$;
\item $Y_F$ is a direct sum of submodules isomorphic to $\mathbb{F}_2$.
\end{itemize}
\end{theorem}


\begin{proof}
Choose $\mathcal{A}$ to be an $\mathbb{F}_2$-basis for $\mathfrak{A}$. By the definition of $\mathfrak{A}$, for each $[f] \in \mathcal{A}$, there exists some $[\gamma_f] \in [K^\times]$ so that $[N_{K/F}(\gamma_f)] = [f]$.  We define $M_{[f]} := \langle [\gamma_f]\rangle$, and observe that $M_{[f]} \simeq \mathbb{F}_2[G]$ and $M_{[f]}^G = \langle [f]\rangle$.  Let $Y_A = \sum_{[f]\in\mathcal{A}}M_{[f]}.$  Observe that $Y_A = \bigoplus_{[f]\in\mathcal{A}}M_{[f]}$ by Lemma \ref{le:exclusion.lemma}, and that $Y_A^G = \bigoplus_{[f] \in \mathcal{A}} \langle [f] \rangle = \langle \mathcal{A} \rangle = \mathfrak{A}$ by construction.

Let $\mathcal{V}$ be an $\mathbb{F}_2$-basis for a complement of $\mathfrak{A}$ in $\mathfrak{V}$. By definition of $\mathfrak{V}$, for each $[f] \in \mathcal{V}$ we can choose $[\gamma_{1,f}],[\gamma_{2,f}] \in [K^\times]$ so that for $\{i,j\} = \{1,2\}$ we have $[\gamma_{i,f}]^{1+\sigma_i} = [f]$ and $[\gamma_{i,f}]^{1+\sigma_j} = [1]$.  For each $[f] \in \mathcal{V}$ we define $M_{[f]} := \langle [\gamma_{1,f}],[\gamma_{2,f}]\rangle$.  We claim that $M_{[f]} \simeq \Omega^1$ and that $M_{[f]}^G = \langle [f]\rangle$.  By construction, the appropriate $\Omega^1$-relations are satisfied for $M_{[f]}$, so we need only check that there are not additional relations.  For this, observe that any nontrivial relation amongst $\{[f],[\gamma_{1,f}],[\gamma_{2,f}]\}$ must involve at least one of $[\gamma_{1,f}]$ or $[\gamma_{2,f}]$ since we know that $[f]$ is nontrivial. On the one hand, if we had a nontrivial relation involving $[\gamma_{1,f}]$, then an application of $1+\sigma_1$ to this relation would tell us that $[f]=[1]$; on the other hand, a nontrivial relation involving $[\gamma_{2,f}]$ would tell us that $[f]=[1]$ after an application of $1+\sigma_2$.  Hence our set is independent, and so $M_{[f]} \simeq \Omega^1$.  This gives $M_{[f]}^G = \langle[f]\rangle$ as well.  Let $Y_V = \sum_{[f]\in\mathcal{V}}M_{[f]}.$ Indeed, we have $Y_V = \bigoplus_{[f]\in\mathcal{V}}M_{[f]}$ by Lemma \ref{le:exclusion.lemma}.  We also have $Y_V^G = \bigoplus_{[f] \in \mathcal{V}} M_{[f]}^G =\bigoplus_{[f] \in \mathcal{V}} \langle [f] \rangle = \langle \mathcal{V} \rangle$ by construction.

Now let $B$ be a complement to $\mathfrak{V}$ within $\mathfrak{B}$, and let $C$ a complement to $\mathfrak{V}$ within $\mathfrak{C}$. Let $B_W$ and $C_W$ be the subspaces defined in Lemma \ref{le:BW.and.CW}.  Let $\mathcal{B}_W$ be an $\mathbb{F}_2$-basis for $B_W$.  For each $[b] \in \mathcal{B}_W$, we know that there exist $[\gamma_1],[\gamma_2],[\gamma_3] \in J(K)$ and $\phi_W([b]) =[c]\in C_W$ which solve Equation (\ref{eq:W.diagram}).  Let $M_{[b]} = \langle [b],[c],[\gamma_1],[\gamma_2],[\gamma_3]\rangle$.  We claim that $M_{[b]} \simeq \Omega^2$.  Certainly the appropriate $\Omega^2$ relations hold by construction, so we simply need to ensure that there are no additional relations.  The elements $[b]$ and $[c]$ are independent since $[b]$ and $[c]$ are each drawn from a complement to $\mathfrak{V} = \mathfrak{B} \cap \mathfrak{C}$ in their respective spaces.  Now if we had a nontrivial $\mathbb{F}_2$-dependence that involved any of $[\gamma_1]$ or $[\gamma_2]$, then an application of $1+\sigma_1$ would force a nontrivial $\mathbb{F}_2$-dependence on $[b]$ and $[c]$, which we've just seen is not possible.  Likewise, a nontrivial $\mathbb{F}_2$-dependence that involves $[\gamma_3]$ would force a nontrivial $\mathbb{F}_2$-dependence between $[b]$ and $[c]$.  Hence the set is independent, and so $M_{[b]} \simeq \Omega^{2}$.  Note this also forces $M_{[b]}^G = \langle[b],[c]\rangle = \langle [b],\phi_W([b])\rangle$.  Let $Y_W = \sum_{[b] \in \mathcal{B}_W} M_{[b]}.$  As before, we in fact have $Y_W = \bigoplus_{[b] \in \mathcal{B}_W} M_{[b]}$, and furthermore $$Y_W^G = \bigoplus_{[b] \in \mathcal{B}_W} M_{[b]}^G = \bigoplus_{[b]\in \mathcal{B}_W} \langle [b], \phi_W([b]) \rangle = B_W \oplus \phi_W(B_W) = B_W \oplus C_W$$ by Lemma \ref{le:BW.and.CW}.

Let $\mathcal{B}_0$ be a basis for a complement to $B_W$ within $B$.  Since $B \subseteq \mathfrak{B}$, each $[f] \in \mathcal{B}_0$ has some $[\gamma_f] \in [K^\times]$ so that $[\gamma_f]^{1+\sigma_1} = [f]$ and $[\gamma_f]^{1+\sigma_2} = [1]$.  Since $[f] \neq [1]$ we get $M_{[f]}:=\langle [\gamma_f]\rangle$ is isomorphic to $\mathbb{F}_2[\overline{G_1}]$, and $M_{[f]}^G = \langle [f]\rangle$.  Let $Y_{B} = \sum_{[f] \in \mathcal{B}_0}M_{[f]}.$  Lemma \ref{le:exclusion.lemma} again gives $Y_B = \bigoplus_{[f] \in \mathcal{B}_0} M_{[f]}$, and furthermore we have $Y_B^G = \bigoplus_{[f] \in \mathcal{B}_0} M_{[f]}^G = \bigoplus_{[f] \in\mathcal{B}_0} \langle [f] \rangle = \langle \mathcal{B}_0\rangle$.

Let $\mathcal{C}_0$ be a basis for a complement to $C_W$ within $C$.  Since $C \subseteq \mathfrak{C}$, for each $[f] \in \mathcal{C}_0$ there exists some $[\gamma_f] \in [K^\times]$ so that $[\gamma_f]^{1+\sigma_2} = [f]$ and $[\gamma_f]^{1+\sigma_1} = [1]$.  Since $[f] \neq [1]$ we get $M_{[f]}:=\langle [\gamma_f]\rangle$ is isomorphic to $\mathbb{F}_2[\overline{G_2}]$, and $M_{[f]}^G = \langle [f]\rangle$.  Let $Y_{C} = \sum_{[f] \in \mathcal{C}_0}M_{[f]}.$ Once again we have $Y_C = \bigoplus_{[f] \in \mathcal{C}_0} M_{[f]}$ by Lemma \ref{le:exclusion.lemma}, and $Y_C^G = \bigoplus_{[f] \in \mathcal{C}_0} M_{[f]}^G = \bigoplus_{[f] \in\mathcal{C}_0} \langle [f] \rangle = \langle \mathcal{C}_0\rangle$.

Let $\mathcal{D}_0$ be a basis for a complement to $(\mathfrak{B}+\mathfrak{C}) \cap \mathfrak{D}$ within $\mathfrak{D}$.  By the definition of $\mathfrak{D}$, for each $[f] \in \mathcal{D}_0$ there exists some $[\gamma_f] \in [K^\times]$ so that $[\gamma_f]^{1+\sigma_2} = [\gamma_f]^{1+\sigma_1} =  [f]$.  Since $[f] \neq [1]$ we get $M_{[f]}:=\langle [\gamma_f]\rangle$ is isomorphic to $\mathbb{F}_2[\overline{G_3}]$, and $M_{[f]}^G = \langle [f]\rangle$.  Let $Y_{D} = \sum_{[f] \in \mathcal{D}_0}M_{[f]}.$ Again we get $Y_D = \bigoplus_{[f] \in \mathcal{D}_0} M_{[f]}$, and $Y_D^G = \bigoplus_{[f] \in \mathcal{D}_0} M_{[f]}^G = \bigoplus_{[f] \in \mathcal{D}_0} \langle [f] \rangle = \langle \mathcal{D}_0\rangle$.

Finally, define $\mathcal{F}_0$ to be a basis for a complement to $\mathfrak{B}+\mathfrak{C}+\mathfrak{D}$ within $[F^\times]$.  For each $[f] \in \mathcal{F}_0$ we define $M_{[f]} = \langle [f]\rangle$, which is clearly isomorphic to $\mathbb{F}_2$.  We let $Y_F = \bigoplus_{[f] \in \mathcal{F}_0} M_{[f]}.$

We have already detailed the fixed parts of each submodule, and we will use this to show that the sum is direct.  First, recall that $Y_A^G = \mathfrak{A}$ and $Y_V^G = \langle \mathcal{V}\rangle$, where $\mathcal{V}$ is chosen to be a complement to $\mathfrak{A}$ in $\mathfrak{V}$.  Then Lemma \ref{le:exclusion.lemma} gives $Y_A+Y_V  = Y_A\oplus Y_V$, and additionally we have $(Y_A\oplus Y_V)^G = \mathfrak{V}$.  

Next, since $Y_W^G = B_W\oplus C_W$ --- where $B_W$ and $C_W$ are complements to $\mathfrak{V}$ in their respective spaces --- Lemma \ref{le:exclusion.lemma} gives $(Y_A \oplus Y_V)+Y_W = Y_A \oplus Y_V \oplus Y_W$, and indeed $(Y_A\oplus Y_V \oplus Y_W)^G = \mathfrak{V} \oplus B_W \oplus C_W$.  

Next, we know that $\mathfrak{B} = \mathfrak{V} \oplus B_W \oplus \langle \mathcal{B}_0\rangle$, and since $Y_B^G = \langle \mathcal{B}_0\rangle$ this means that $Y_A\oplus Y_V \oplus Y_W + Y_B = Y_A\oplus Y_V \oplus Y_W \oplus Y_B$, and $(Y_A\oplus Y_V \oplus Y_W \oplus Y_B)^G = \mathfrak{B} \oplus C_W$.  Using the facts that $\mathfrak{B} \cap \mathfrak{C} = \mathfrak{V}$, that $Y_C^G = \langle \mathcal{C}_0\rangle$, and that $\mathfrak{C} = \mathfrak{V} \oplus C_W \oplus \langle \mathcal{C}_0 \rangle$, Lemma \ref{le:exclusion.lemma} gives $Y_A\oplus Y_V \oplus Y_W \oplus Y_B + Y_C= Y_A \oplus Y_V \oplus Y_W \oplus Y_B \oplus Y_C$, and $(Y_A \oplus Y_V \oplus Y_W \oplus Y_B \oplus Y_C)^G = \mathfrak{B}+\mathfrak{C}$.  

For the next term, since $Y_D^G = \langle \mathcal{D}_0 \rangle$, where $\mathcal{D}_0$ is a complement to $(\mathfrak{B}+\mathfrak{C})\cap \mathfrak{D}$, Lemma \ref{le:exclusion.lemma} gives us $Y_A\oplus Y_V \oplus Y_W \oplus Y_B \oplus Y_C + Y_D = Y_A \oplus Y_V \oplus Y_W \oplus Y_B \oplus Y_C \oplus Y_D$, and indeed $(Y_A \oplus Y_V \oplus Y_W \oplus Y_B \oplus Y_C \oplus Y_D)^G = \mathfrak{B}+\mathfrak{C}+\mathfrak{D}$.  

Finally, since $Y_F^G = \langle \mathcal{F}_0\rangle$, where $\mathcal{F}_0$ is a complement to $\mathfrak{B}+\mathfrak{C}+\mathfrak{D}$ in $[F^\times]$, one final application of Lemma \ref{le:exclusion.lemma} gives $\hat J^G = [F^\times]$ and $$\hat J = Y_A \oplus Y_V \oplus Y_W \oplus Y_B \oplus Y_C \oplus Y_D \oplus Y_F.$$
\end{proof}

\begin{corollary}\label{cor:solvable.in.hat.J}
Suppose that $[f] \in \mathfrak{M}$ for $\mathfrak{M} \in \{\mathfrak{A},\mathfrak{B},\mathfrak{C},\mathfrak{D}\}$. Then the diagram corresponding to $\mathfrak{M}$ has a solution for $[f]$ in which each term of the solution comes from $\hat J$.
\end{corollary}

\begin{proof}
Suppose first that $[f] \in \mathfrak{A}$.  Since $\mathcal{A}$ is a basis for $\mathfrak{A}$, we have $[f] = \prod_{i=1}^n [f_i]$ for appropriately chosen $[f_i] \in \mathcal{A}$. By construction, there exist $[\gamma_i] \in J(K)$ so that $[N_{K/F}(k_i)]=[f_i]$, and so we get $k = \prod_{i=1}^n [k_i]$ has $[N_{K/F}(k)] = [f]$. Hence the diagram corresponding to $\mathfrak{A}$ is solvable for $[f]$ with terms drawn from $\hat J$.


Now suppose that $[f] \in \mathfrak{B}$.  Since we know $\mathfrak{B} = \mathfrak{V} \oplus B_W \oplus \langle \mathcal{B}_0\rangle$, we can write $[f] = \prod_{i=1}^n [f_i] \prod_{j=1}^m [\widehat{f_j}] \prod_{k=1}^\ell [\widetilde{f_k}]$, where $f_i \in \mathcal{A} \cup \mathcal{V}$, $\widehat{f_j} \in \mathcal{B}_W$ and $\widetilde{f_k} \in \mathcal{B}_0$ are appropriately chosen.  Based on the construction of the terms from $Y_A$,$Y_V$, $Y_W$, and $Y_B$, we have elements $[\gamma_i],[\widehat{\gamma_j}],[\widetilde{\gamma_k}] \in \hat J$ that solve the diagram corresponding to $\mathfrak{B}$.  Hence if we let $[\gamma] = \prod_{i=1}^n [\gamma_i] \prod_{j=1}^m [\widehat{\gamma_j}] \prod_{k=1}^\ell [\widetilde{\gamma_k}]$, then the diagram corresponding to $\mathfrak{B}$ is solved with $[\gamma]$ and $[f]$.  An analogous argument settles the case where $[f] \in \mathfrak{C}$.

We have left to settle the statement for $\mathfrak{M}=\mathfrak{D}$.  This will take a bit more work.  Since $\mathcal{D}_0$ is a basis for a complement to $(\mathfrak{B}+\mathfrak{C})\cap \mathfrak{D}$ within $\mathfrak{D}$, we can write $[f] = [\hat f] \prod_{i=1}^n[f_i]$, where $[\hat f] \in (\mathfrak{B}+\mathfrak{C})\cap \mathfrak{D}$, and with $[f_i] \in \mathcal{D}_0$ appropriately chosen.  By construction of $Y_D$, for each $i$ we have an element $[\gamma_i]$ so that the diagram for $\mathfrak{D}$ is solved with $[\gamma_i]$ and $[f_i]$.  Furthermore, since $[\hat f] \in (\mathfrak{B}+\mathfrak{C}) \cap \mathfrak{D}$, Lemma \ref{le:characterizing.W.elements} tells us that $[\hat f] = [b][c]$ has the property that Equation (\ref{eq:W.diagram}) is solvable for $[b]$ and $[c]$:
$$
\begin{tikzpicture}[scale=0.65]
\node (Ut1) at (-4,0) {$[\widehat{\gamma_1}]$};
\node (Ut2) at (0,0) {$[\widehat{\gamma_2}]$};
\node (Ut3) at (4,0) {$[\widehat{\gamma_3}]$};
\node (Ub1) at (-2,-2) {$[b]$};
\node (Ub2) at (2,-2) {$[c]$};
\node (Ub0) at (-6,-2) {$[1]$};
\node (Ub3) at (6,-2) {$[1].$};

\draw[->] (Ut1) to (Ub0);
\draw[->] (Ut1) to (Ub1);
\draw[->] (Ut2) to (Ub1);
\draw[->] (Ut2) to (Ub2);
\draw[->] (Ut3) to (Ub2);
\draw[->] (Ut3) to (Ub3);
\end{tikzpicture}
$$
Since $B$ is a complement to $\mathfrak{V}$ in $\mathfrak{B}$ and $C$ is a complement to $\mathfrak{V}$ in $\mathfrak{C}$, we can write  $[b] = [v_1][b_W]$ and $[c] = [v_2][c_W]$ for $[v_1],[v_2] \in \mathfrak{V}$ and $[b_W] \in B$ and $[c_W] \in C$.  For $i \in \{1,2\}$, the construction of $Y_A$ and $Y_V$ give $[\widetilde{\gamma_{L,i}}],[\widetilde{\gamma_{R,i}}] \in \hat J$ that accompany $[v_i]$ in solving the diagram for $\mathfrak{V}$.  Note in particular this means that the diagram corresponding to $\mathfrak{D}$ is solved for $[\widetilde{\gamma_{L,i}}][\widetilde{\gamma_{R,i}}]$ and $[v_i]$.  We have left to deal with the $[b_W]$ and $[c_W]$ terms.

From our previous equations we get
$$
\begin{tikzpicture}[scale=0.65]
\node (Ut1) at (-4,0) {$[\widehat{\gamma_1}][\widetilde{\gamma_{L,1}}]$};
\node (Ut2) at (0,0) {$[\widetilde{\gamma_{R,1}}][\widehat{\gamma_2}][\widetilde{\gamma_{L,2}}]$};
\node (Ut3) at (4,0) {$[\widetilde{\gamma_{R,2}}][\widehat{\gamma_3}]$};
\node (Ub1) at (-2,-2) {$[b_W]$};
\node (Ub2) at (2,-2) {$[c_W]$};
\node (Ub0) at (-6,-2) {$[1]$};
\node (Ub3) at (6,-2) {$[1].$};

\draw[->] (Ut1) to (Ub0);
\draw[->] (Ut1) to (Ub1);
\draw[->] (Ut2) to (Ub1);
\draw[->] (Ut2) to (Ub2);
\draw[->] (Ut3) to (Ub2);
\draw[->] (Ut3) to (Ub3);
\end{tikzpicture}
$$ This means that $[b_W] \in B_W$, and by Lemma \ref{le:BW.and.CW} we have $\phi_W([b_W]) = c_W \in C_W$.  Hence $[b_W] = \prod_{j=1}^m [b_j]$ for $[b_j]$ appropriately chosen from $\mathcal{B}_W$.  By the construction of $Y_W$ we have elements $[\gamma_{1,j}],[\gamma_{2,j}],[\gamma_{3,j}] \in \hat J$ which solve Equation \ref{eq:W.diagram} for $[b_j]$ and $\phi_W([b_j])$.  Hence $\prod_{j=1}^m [\gamma_{1,j}]$, $\prod_{j=1}^m[\gamma_{2,j}]$, and $\prod_{j=1}^m [\gamma_{3,j}]$ solve Equation \ref{eq:W.diagram} for $[b_W]$ and $\prod_{j=1}^m \phi_W([b_j]) = \phi_W(\prod_{j=1}^m [b_j]) = [c_W]$. In particular, the equation corresponding to $\mathfrak{D}$ is solved by $\prod_{j=1}^m [\gamma_{1,j}][\gamma_{2,j}][\gamma_{3,j}]$ and $[b_W][c_W]$.  

In all, our original element $[f] \in \mathfrak{D}$ has now been expressed as $[f] = [\hat f]\prod_{i=1}^n[f_i] = [b][c] \prod_{i=1}^n [f_i] = [v_1][b_W][v_2][c_W]\prod_{i=1}^n[f_i]$, where each of $[v_1],[v_2],[b_W][c_W],$ and $\prod_{i=1}^n[f_i]$ have some corresponding element $[\gamma] \in \hat J$ which solves the diagram corresponding to $\mathfrak{D}$.
\end{proof}


\section{A module whose fixed part complements $[F^\times]$ in $J(K)^G$}\label{sec:X.candidate}

Lemma \ref{le:exclusion.lemma} tells us that independent summands of $J(K)$ have independent fixed parts.  Since we've already constructed a module whose fixed part is $[F^\times]$, we now are interested in finding a complementary module whose fixed part spans a complement to $[F^\times]$ in $J(K)^G$ --- at least to the degree that such a goal is achievable at all.  Ultimately this search will culminate in Theorem \ref{th:construction.of.X} at the end of this section, but to work towards this result we must first determine precisely which elements from $J(K)^G$ come from $[F^\times]$.  

Kummer theory tells us that we can determine whether an element $[\gamma] \in J(K)^G$ comes from $[F^\times]$ by examining the Galois group of the extension it generates over $F$: $$[\gamma] \in J(K)^G \cap [F^\times]\setminus
\{[1]\} \Leftrightarrow \Gal(K(\sqrt{\gamma})/F) \simeq \mathbb{Z}/2\mathbb{Z}\oplus\mathbb{Z}/2\mathbb{Z}\oplus\mathbb{Z}/2\mathbb{Z}.$$
The following result gives a slightly more nuanced view of this phenomenon.  Note that in this result --- and hence for much of the duration of this section --- we use the notation $[\gamma]_i$ to indicate the class of an element $\gamma \in K_i^\times \cap K^{\times 2}$ considered in the set $(K_i^\times \cap K^{\times 2})/K_i^{\times 2}$ for $i \in \{1,2,3\}$.

\begin{lemma}\label{le:Galois.groups.for.fixed.elements}
Suppose that $[\gamma] \in J(K)^G \setminus\{[1]\}$.  Then $K(\sqrt{\gamma})/F$ is Galois, and if $\hat \sigma_1$ and $\hat\sigma_2$ represent lifts of $\sigma_1,\sigma_2 \in \Gal(K/F)$ to the group $\Gal(K(\sqrt{\gamma})/F)$, then we have
\begin{align*}
[N_{K/K_1}(\gamma)]_1 = [1]_1 &\Leftrightarrow \hat\sigma_2^2 = \text{id}\\
[N_{K/K_2}(\gamma)]_2 = [1]_2 &\Leftrightarrow \hat\sigma_1^2 = \text{id}\\
[N_{K/K_3}(\gamma)]_3 = [1]_3 &\Leftrightarrow \left(\hat\sigma_1\hat\sigma_2\right)^2 = \text{id}.
\end{align*}
\end{lemma}

\begin{proof}
We consider the first statement first.  Observe that we already know that $\hat\sigma_2^2$ acts trivially on $\sqrt{a_1}$ and $\sqrt{a_2}$, so we only need to determine the action of $\hat \sigma_2^2$ on $\sqrt{\gamma}$.  For this, note that
$$\sqrt{\gamma}^{\hat\sigma_2^2-1} = \left(\sqrt{\gamma}^{\hat \sigma_2+1}\right)^{\hat \sigma_2-1} = \left(\pm\sqrt{\gamma^{\sigma_2+1}}\right)^{\hat \sigma_2-1}.$$ Since $[\gamma] \in J(K)^G$ we have that $[\gamma]^{\sigma_2+1} = [N_{K/K_1}(\gamma)]=[1]$.  Hence we have $N_{K/K_1}(\gamma) \in K_1^\times \cap K^{\times 2}$, and by Kummer theory this means that $N_{K/K_1}(\gamma) = a_2^{\varepsilon}k_1^2$ for some $\varepsilon \in \{0,1\}$ and $k_1 \in K_1^\times$. Note that $\varepsilon = 0$ if and only if $N_{K/K_1}(\gamma) \in K_1^{\times 2}$, which is equivalent to $[N_{K/K_1}(\gamma)]_1 = [1]_1$. Hence our previous calculation continues
$$\sqrt{\gamma}^{\hat\sigma_2^2-1} = \left(\pm \sqrt{a_2}^{\varepsilon}k_1\right)^{\hat \sigma_2-1} =  \left(\pm \sqrt{a_2}^{\varepsilon}k_1\right)^{\sigma_2-1} = (-1)^\varepsilon.$$This gives the desired result.

Similar calculations give the other two results.
\end{proof}

\begin{corollary}\label{le:supernorm.as.F.detector}
Define $T:J(K)^G \to \bigoplus_{i=1}^3 (K_i^\times\cap K^{\times 2})/K_i^{\times 2}$ by $$T([\gamma]) = ([N_{K/K_1}(\gamma)]_1,[N_{K/K_2}(\gamma)]_2,[N_{K/K_3}(\gamma)]_3).$$  Then $\ker(T)=[F^\times]$.
\end{corollary}

\begin{remark*}
Note that Kummer theory tells us that each $(K_i^\times \cap K^{\times 2})/K_i^{\times 2}$  consists of only two distinct classes, with representatives drawn from $\{1,a_1,a_2,a_1a_2\}$.  For example $(K_3^\times \cap K^{\times 2})/K_3^{\times 2}$ has $[1]_3=[a_1a_2]_3$ and $[a_1]_3=[a_2]_3$ as its elements. 
For the sake of lightening what would otherwise be fairly weighty notation, when considering elements in the image of $T$ we will  suppress the bracket notation in its coordinates; that is to say, if $T([\gamma]) = ([u]_1,[v]_2,[w]_3)$, then we will instead write $T([\gamma]) = (u,v,w)$.
\end{remark*}

Our goal, then, is to build a module whose fixed part spans the image of $T$, ideally while avoiding $[F^\times]$ as much as possible.  The first question we consider when looking for such a module is to determine when elements with a nontrivial image under $T$ are themselves in the image of either $1+\sigma_1$ or $1+\sigma_2$.  We start with

\begin{lemma}\label{le:indices.when.in.the.image}
Suppose $[N_{K/F}(\gamma)]= [1]$.  Then $[N_{K/K_1}(\gamma)],[N_{K/K_2}(\gamma)] \in J(K)^G$, and
\begin{itemize}
\item $[N_{K/F}(\gamma)]_F = [1]_F \Leftrightarrow T([N_{K/K_1}(\gamma)]) = (1,1,1) \Leftrightarrow T([N_{K/K_2}(\gamma)]) = (1,1,1)$;
\item $[N_{K/F}(\gamma)]_F = [a_1]_F \Leftrightarrow T([N_{K/K_1}(\gamma)]) = (1,a_1,a_1) \Leftrightarrow T([N_{K/K_2}(\gamma)]) = (1,1,a_1)$;
\item $[N_{K/F}(\gamma)]_F = [a_2]_F \Leftrightarrow T([N_{K/K_1}(\gamma)]) = (1,1,a_1) \Leftrightarrow T([N_{K/K_2}(\gamma)]) = (a_2,1,a_1)$; and
\item $[N_{K/F}(\gamma)]_F = [a_1a_2]_F \Leftrightarrow T([N_{K/K_1}(\gamma)]) = (1,a_1,1) \Leftrightarrow T([N_{K/K_2}(\gamma)]) = (a_2,1,1)$.
\end{itemize}
\end{lemma}

\begin{proof}
Observe first that since $[N_{K/F}(\gamma)] = [1]$, Kummer theory tells us that $[N_{K/F}(\gamma)]_F \in \{[1]_F,[a_1]_F,[a_2]_F,[a_1a_2]_F\}$.  So let us write $N_{K/F}(\gamma) = f^2 a_1^{\varepsilon_1}a_2^{\varepsilon_2}$.  The result then follows from the following calculations:
\begin{align*}
[N_{K/K_1}(N_{K/K_1}(\gamma))]_1 &= [N_{K/K_1}(\gamma)^2]_1 = [1]_1\\
[N_{K/K_2}(N_{K/K_1}(\gamma))]_2 &= [N_{K/F}(\gamma)]_2 = [f^2a_1^{\varepsilon_1}a_2^{\varepsilon_2}]_2 = [a_1]_2^{\varepsilon_1}\\
[N_{K/K_3}(N_{K/K_1}(\gamma))]_3 &= [N_{K/F}(\gamma)]_3 = [f^2a_1^{\varepsilon_1}a_2^{\varepsilon_2}]_3 = [a_1]_3^{\varepsilon_1+\varepsilon_2}\\
[N_{K/K_1}(N_{K/K_2}(\gamma))]_1 &= [N_{K/F}(\gamma)]_1 = [f^2a_1^{\varepsilon_1}a_2^{\varepsilon_2}]_1 = [a_2]_1^{\varepsilon_2}\\
[N_{K/K_2}(N_{K/K_2}(\gamma))]_2 &= [N_{K/K_2}(\gamma)^2]_2 = [1]_2\\
[N_{K/K_3}(N_{K/K_2}(\gamma))]_3 &= [N_{K/F}(\gamma)]_3 = [f^2a_1^{\varepsilon_1}a_2^{\varepsilon_2}]_3 = [a_1]_3^{\varepsilon_1+\varepsilon_2}\qedhere
\end{align*}
\end{proof}

\begin{corollary}\label{cor:dimension.two.fixed.parts.are.from.F}
Suppose that $[\gamma] \in J(K)$ generates a module isomorphic to $\mathbb{F}_2[\overline{G_i}]$ for some $i \in \{1,2,3\}$.  Then $T(\langle[\gamma]\rangle^G) = \{(1,1,1)\}$.  
\end{corollary}

\begin{proof}
We proceed by cases.  If $\langle[\gamma]\rangle \simeq \mathbb{F}_2[\overline{G_1}]$, then we have $[\gamma]^{1+\sigma_2} = [1]$, so that $[N_{K/F}(\gamma)] = [1]$.  Since $\langle[\gamma]\rangle^G = \langle [\gamma]^{1+\sigma_1}\rangle$ in this case, our objective is to show that $T([\gamma]^{1+\sigma_1}) = (1,1,1)$.  But since $[\gamma]^{1+\sigma_1} = [N_{K/K_2}(\gamma)]$, the previous lemma tells us that if we have $T([\gamma]^{1+\sigma_1}) \neq (1,1,1)$ then we have $(1,1,1) \neq T([N_{K/K_1}(\gamma)]) = T([\gamma]^{1+\sigma_2}) = T([1])$ as well, a contradiction.  The same argument gives the result for $i=2$.

For $i=3$, a variation on this argument works: we know we have $[N_{K/K_1}(\gamma)] = [N_{K/K_2}(\gamma)]$, and yet the lemma above provides no nontrivial case in which $T([N_{K/K_1}(\gamma)]) = T([N_{K/K_2}(\gamma)])$.
\end{proof}

Lemma \ref{le:indices.when.in.the.image} tells us a relationship between the possible values under $T$ for elements from $J(K)^G$ which are in the image of a common element; if two elements $[x]$ and $[y]$ have ``compatible" images under $T$ (i.e., allowable in light of Lemma \ref{le:indices.when.in.the.image}), is it the case that there exists some $[\gamma]$ so that $[x] = [N_{K/K_1}(\gamma)]$ and $[y] = [N_{K/K_2}(\gamma)]$?  The answer to this is generally ``no," but there is a weaker version which we will take advantage of.

\begin{lemma}\label{le:compatible.supernorms.means.preimage.exists.somewhere}
Suppose that $[x],[y]\in J(K)^G$ are given, and that $$(T([x]),T([y])) \in \left\{\left((1,a_1,a_1),(1,1,a_1)\right),\left((1,1,a_1),(a_2,1,a_1)\right),\left((1,a_1,1),(a_1,1,1)\right)\right\}.$$  Then there exists some $[\gamma]$ with $[N_{K/F}(\gamma)]=[1]$ so that $T([N_{K/K_1}(\gamma)]) = T([x])$ and $T([N_{K/K_2}(\gamma)]) = T([y])$.
\end{lemma}

\begin{proof}
Our approach will be to argue that the appearance of these elements in the image of $T$ guarantees the solvability of certain embedding problems, from which we deduce the solvability of certain equations involving norms. 

We first handle the case where $\im(T)$ contains $\{(1,a_1,a_1),(1,1,a_1)\}$.  Since $T([x]) = (1,a_1,a_1)$, we know from Lemma \ref{le:Galois.groups.for.fixed.elements} that in $K(\sqrt{x})/F$ the generators $\sigma_1,\sigma_2 \in \Gal(K/F)$ extend to elements $\hat\sigma_1,\hat\sigma_2 \in \Gal(K(\sqrt{x})/F)$ which satisfy the relations $$\hat\sigma_2^2 = \hat \sigma_1^4 = (\hat\sigma_1\hat\sigma_2)^4 = \text{id}.$$  Hence $\Gal(K(\sqrt{x})/F)\twoheadrightarrow \Gal(K/F)$ solves the embedding problem $\mathbb{Z}/4\mathbb{Z}\oplus\mathbb{Z}/2\mathbb{Z} \twoheadrightarrow\mathbb{Z}/2\mathbb{Z}\oplus\mathbb{Z}/2\mathbb{Z}$, and in particular $K_1/F$ embeds in a cyclic extension of degree $4$.  By \cite[Theorem 3]{Albert} we have $-1\in N_{K_1/F}(K_1^\times)$; since we have $-a_1 = (\sqrt{a_1})^{1+\sigma_1}=N_{K_1/F}(\sqrt{a_1})$, it therefore follows that $a_1 \in N_{K_1/F}(K_1^\times)$, say $a_1 = N_{K_1/F}(k_1)$ for $k_1 \in K_1^\times$.

On the other hand, since $T([y]) = (1,1,a_1)$, we know from Lemma \ref{le:Galois.groups.for.fixed.elements} that the generators $\sigma_1,\sigma_2 \in \Gal(K/F)$ extend to elements $\tilde\sigma_1,\tilde\sigma_2 \in \Gal(K(\sqrt{y})/F)$ that satisfy $$\tilde\sigma_2^2=\tilde\sigma_1^2 = (\tilde\sigma_1\tilde\sigma_2)^4 = \text{id}.$$  From this we see that $\Gal(K(\sqrt{y})/F) \twoheadrightarrow \Gal(K/F)$ solves the embeddling problem $D_4 \twoheadrightarrow \mathbb{Z}/2\mathbb{Z}\oplus\mathbb{Z}/2\mathbb{Z}$, where the kernel of the latter surjection is $\langle(\tilde\sigma_1\tilde\sigma_2)^2\rangle$.  By a well-known result for the solvability of such embedding problems (see, e.g., \cite[Proposition III.3.3]{JY}), we therefore have $a_1 \in N_{K_2/F}(K_2^\times)$, say $a_1 = N_{K_2/F}(k_2)$ for $k_2 \in K_2^\times$.  

By \cite[Lemma 2.14]{Wadsworth}, there exists some $\gamma \in K^\times$ and $f \in F^\times$ so that $N_{K/K_1}(\gamma) = fk_1$ and $N_{K/K_2}(\gamma) = fk_2$.  In particular we have $[N_{K/F}(\gamma)] = [f^2a_1] = [1]$, and so $[N_{K/F}(\gamma)]_F = [a_1]_F$.  An application of Lemma \ref{le:indices.when.in.the.image} finishes this case.

The second case is effectively identical to the first.  For the last case, note that the images we're given provide two $D_4$-extensions over $K/F$, one in which $\sigma_1$ extends to an element of order four, and another where $\sigma_2$ extends to an element of order $4$.  In the former case we then get $a_1a_2 \in N_{K_2/F}(K_2^\times)$, and in the latter we get $a_1a_2 \in N_{K_1/F}(K_1^\times)$.  From here the proof proceeds as before.
\end{proof}

\begin{remark*}
The use of \cite[Lemma 2.14]{Wadsworth}  amounts to an appeal to Hilbert 90 for biquadratic extensions.  See \cite{DMSS}.
\end{remark*}

We are now prepared for the main result of this section.

\begin{theorem}\label{th:construction.of.X}
There exists $X \subseteq J(K)$ with $T(X^G) = \im(T)$, so that
$$
X \simeq \left\{\begin{array}{ll}
\{[1]\},&\text{ if $\dim(\im(T)) = 0$}\\
\mathbb{F}_2,&\text{ if $\dim(\im(T)) = 1$}\\
\Omega^{-1},&\text{ if $\dim(\im(T)) = 2$ and $\im(T)$ is one of the ``coordinate planes"}\\
\mathbb{F}_2\oplus \mathbb{F}_2,&\text{ if $\dim(\im(T)) = 2$ and $\im(T)$ is not one of the ``coordinate planes"}\\
\Omega^{-2},&\text{ if $\dim(\im(T)) = 3$ and $T\left([N_{K/K_1}(K^\times)] \cap [N_{K/K_2}(K^\times)]\right) \neq \{(1,1,1)\}$}\\
\Omega^{-1}\oplus\Omega^{-1},&\text{ if $\dim(\im(T)) = 3$ and $T\left([N_{K/K_1}(K^\times)] \cap [N_{K/K_2}(K^\times)]\right) = \{(1,1,1)\}$.}
\end{array}\right.
$$
In all cases except the last, we have $[F^\times] \cap X^G = \{[1]\}$; in the last case we have $\dim(X^G \cap [F^\times]) = 1$ and $X^G \cap (\mathfrak{B}+\mathfrak{C}+\mathfrak{D}) =\{[1]\}$.
\end{theorem}

\begin{proof}
We proceed by cases based on $\dim(\im(T))$.  First, if $\dim(\im(T)) = 1$ then let $[x] \in J(K)^G$ be given so that $T([x]) \neq (1,1,1)$.  Then $X:=\langle[x]\rangle$ has the desired properties.

Now suppose that $\dim(\im(T)) = 2$.  By Lemma \ref{le:compatible.supernorms.means.preimage.exists.somewhere} we know that if $\im(T)$ is any of 
\begin{align*}
\langle (1,a_1,a_1),(1,1,a_1)\rangle &= \{(1,w,z):w \in (K_2^\times\cap K^{\times 2})/K_2^{\times 2},z \in (K_3^\times \cap K^{\times 2})/K_3^{\times 2}\}\\
\langle (1,1,a_1),(a_2,1,a_1)\rangle &= \{(w,1,z):w \in (K_1^\times\cap K^{\times 2})/K_1^{\times 2},z \in (K_3^\times \cap K^{\times 2})/K_3^{\times 2}\}\text{, or}\\
\langle (1,a_1,1),(a_1,1,1)\rangle &= \{(w,z,1):w \in (K_1^\times\cap K^{\times 2})/K_1^{\times 2},z \in (K_2^\times \cap K^{\times 2})/K_2^{\times 2}\},
\end{align*} then we can find the some $[\gamma] \in J(K)$ so that $\left\langle T\left([N_{K/K_1}(\gamma)]\right),T\left([N_{K/K_2}(\gamma)]\right)\right\rangle = \im(T)$.  It is easy to see that $X:=\langle[\gamma]\rangle \simeq \Omega^{-1}$, and since the nontrivial fixed elements in this module have nontrivial images under $T$, we get $X^G \cap [F^\times] = \{[1]\}$.

On the other hand, if $\dim(\im(T)) = 2$ but $\im(T)$ is none of the three subspaces above, then Lemma \ref{le:indices.when.in.the.image} tells us that no element from $J(K)^G \setminus \ker(T)$ is in the image of $1+\sigma_1$ or $1+\sigma_2$.  In this case we let $[x_1],[x_2] \in J(K)^G$ be given so that $\{T([x_1]),T([x_2])\}$ forms a basis for $\im(T)$; we then get $X:=\langle[x_1],[x_2]\rangle\simeq \mathbb{F}_2\oplus\mathbb{F}_2$, with $X \cap [F^\times] = \{[1]\}$.

Now suppose that $\dim(\im(T)) = 3$.  First consider the case where $T([N_{K/K_1}(K^\times)]\cap[N_{K/K_2}(K^\times)]) \neq \{(1,1,1)\}$, and let $[x]$ be given so that $[N_{K/K_1}(\gamma_2)] = [x] = [N_{K/K_2}(\gamma_1)]$ for some $[\gamma_1],[\gamma_2] \in J(K)$ and with $T([x]) \neq (1,1,1)$. (Note that since $[x]$ is in the image of $N_{K/K_1}$ and $N_{K/K_2}$, it is automatically in $J(K)^G$; hence it makes sense to evaluate its image under $T$.) Lemma \ref{le:indices.when.in.the.image} tells us that $T([x]) = (1,1,a_1)$, and furthermore that $T([N_{K/K_1}(\gamma_1)]) = (1,a_1,a_1)$ and $T([N_{K/K_2}(\gamma_2)]) = (a_2,1,a_1)$.  We claim that $X:=\langle[\gamma_1],[\gamma_2]\rangle \simeq \Omega^{-2}$; certainly the appropriate relations hold, so we only need to check that the module is $5$-dimensional.  Note that $\{[N_{K/K_1}(\gamma_1)],[x],[N_{K/K_2}(\gamma_2)]\}$ must be independent since their images under $T$ are independent, and hence any nontrivial dependence must involve $[\gamma_1]$ or $[\gamma_2]$.  But an application of $1+\sigma_1$ (or $1+\sigma_2$) to such a relation creates a nontrivial relation amongst $\{[N_{K/K_1}(\gamma_1)],[x],[N_{K/K_2}(\gamma_2)]\}$, contrary to their independence.  Since we have $X \simeq \Omega^{-2}$, we get $X^G = \langle [N_{K/K_1}(\gamma_1)],[x],[N_{K/K_2}(\gamma_2)]\rangle$, whence $X^G \cap [F^\times] = \{[1]\}$.

Alternatively, suppose that $\dim(\im(T)) = 3$ but $T([N_{K/K_1}(K^\times) \cap [N_{K/K_2}(K^\times)]) = \{(1,1,1)\}$.  Lemma \ref{le:compatible.supernorms.means.preimage.exists.somewhere} gives us elements $[\gamma_1],[\gamma_2] \in J(K)$ so that \begin{align*}
T([N_{K/K_1}(\gamma_1)])&=(1,a_1,a_1)\\
T([N_{K/K_2}(\gamma_1)])&=(1,1,a_1) = T([N_{K/K_1}(\gamma_2)])\\
T([N_{K/K_2}(\gamma_2)])&=(a_2,1,a_1).
\end{align*}
We define $X = \langle [\gamma_1],[\gamma_2]\rangle$.  One sees that $\langle[\gamma_1]\rangle \simeq \langle[\gamma_2]\rangle \simeq \Omega^{-1}$ in the same manner as above (these modules satisfy the appropriate relations by definition, and one can argue they generate a module of the appropriate dimension by leveraging the independence of the image of their fixed components under $T$).  We claim that $X \simeq \Omega^{-1}\oplus\Omega^{-1}$; for the sake of contradiction, then, assume instead that $\langle[\gamma_1]\rangle \cap \langle[\gamma_2]\rangle \neq \{[1]\}$.  By Lemma \ref{le:exclusion.lemma} this implies that there is some $[x]\neq [1]$ with $[x] \in \langle[\gamma_1]\rangle^G \cap \langle[\gamma_1]\rangle^G
$.  Considering images under $T$ and using Lemma \ref{le:indices.when.in.the.image}, we must have $[N_{K/K_2}(\gamma_1)] = [x] = [N_{K/K_1}(\gamma_2)]$, contrary to the assumption in this case that $T([N_{K/K_1}(K^\times)] \cap [N_{K/K_2}(K^\times)]) = \{(1,1,1)\}$.  Hence we get $X \simeq \Omega^{-1} \oplus \Omega^{-1}$.

Finally, we check that $\dim(X^G \cap [F^\times]) = 1$ with $X^G \cap (\mathfrak{B}+\mathfrak{C}+\mathfrak{D}) = \{[1]\}$.  The former follows from the rank-nullity theorem applied to the function $T$; in fact we see that $X^G \cap [F^\times] = \{[1],[N_{K/K_2}(\gamma_1)][N_{K/K_1}(\gamma_2)]\}$.  For the latter, suppose instead that $[N_{K/K_2}(\gamma_1)][N_{K/K_1}(\gamma_2)] \in \mathfrak{B}+\mathfrak{C}+\mathfrak{D}$.  Then we get $[N_{K/K_2}(\gamma_1)][N_{K/K_1}(\gamma_2)] = [f_\mathfrak{B}][f_{\mathfrak{C}}][f_\mathfrak{D}]$ for some $[f_\mathfrak{B}]\in\mathfrak{B},[f_\mathfrak{C}]\in\mathfrak{C}$ and $[f_\mathfrak{D}]\in\mathfrak{D}$; in particular, this means we have elements elements $[\gamma_{\mathfrak{B}}],[\gamma_{\mathfrak{C}}],[\gamma_{\mathfrak{D}}] \in J(K)$ which solve the relevant diagrams from Figure \ref{fig:first.filtration.diagrams}.  One can then check that
\begin{align*}
N_{K/K_1}([\gamma_2][\gamma_{\mathfrak{C}}]) = [N_{K/K_1}(\gamma_2)][f_{\mathfrak{C}}] = [N_{K/K_2}(\gamma_1)][f_{\mathfrak{B}}][f_{\mathfrak{D}}] = N_{K/K_2}([\gamma_1][\gamma_{\mathfrak{B}}][\gamma_{\mathfrak{D}}]).
\end{align*} This element is conspicuously an element in $[N_{K/K_1}(K^\times)] \cap [N_{K/K_2}(K^\times)]$, and since $\ker(T) = [F^\times]$ we get that $T([N_{K/K_2}(\gamma_1)][f_\mathfrak{B}][f_\mathfrak{D}]) = T([N_{K/K_2}(\gamma_1)]) \neq \{(1,1,1)\}$.  This runs contrary to the overriding assumption in this case, that $T([N_{K/K_1}(K^\times)] \cap [N_{K/K_2}(K^\times)]) = \{(1,1,1)\}$.
\end{proof}

\section{Proof of Theorem \ref{th:main.theorem}}\label{sec:proof.of.main.result}

We need one final preparatory result, which is again a manifestation of Hilbert 90 in the biquadratic case. 

\begin{lemma}\label{le:chain.breaking.lemma}
Let $\{\ell,m,n\}=\{1,2,3\}$.  If $f \in F^\times$ has $[f] \in [N_{K/K_\ell}(K^\times)]$, then $[f] \in [N_{K_m/F}(K_m^\times)][N_{K_n/F}(K_n^\times)]$.  
\end{lemma}


\begin{proof}
We prove the result when $\ell = 3$, $m = 1$ and $n=2$; the other results follow by the symmetry of the fields $K_1,K_2$ and $K_3$.

First, we argue that if $f \in F^\times$ has $[f] \in [N_{K/K_3}(K^\times)]$, then \begin{equation}\label{eq:first.fact}\frac{f}{a_1^\varepsilon} = N_{K/K_3}(\tilde k)\end{equation}for some $\tilde k \in K^\times$ and $\varepsilon \in \{0,1\}$. To see this, note that $f = k^{1+\sigma_1 \sigma_1} \hat k^2$ for some $k,\hat k \in K$.  Solving for $\hat k^2$ and using the fact that $F \subseteq K_3$, we then have $\hat k^2 \in K_3$.  But this means  $\hat k^2 \in K^{\times 2} \cap K_3^\times$, so by Kummer theory we get $\hat k^2=k_3^2 a_1^\varepsilon$, where $k_3 \in K_3^\times$ and $\varepsilon \in \{0,1\}$.  Naturally we have $k_3^2 = N_{K/K_3}(k_3)$, so that our original expression becomes $$f = k^{1+\sigma_1\sigma_2}\hat k^2 = k^{1+\sigma_1\sigma_2} k_3^2a_1^\varepsilon = N_{K/K_3}(kk_3) a_1^\varepsilon.$$ Setting $\tilde k = k k_3$ and dividing through by $a_1^\varepsilon$ gives Equation (\ref{eq:first.fact}).

Now we argue that \begin{equation}\label{eq:second.fact}F^\times \cap N_{K/K_3}(K^\times) \subseteq  N_{K_1/F}(K_1^\times) \cdot N_{K_2/F}(K_2^\times).\end{equation}  For this, suppose that we have elements $g \in F^\times$ and $k \in K^\times$ so that $g = N_{K/K_3}(k)$. Now $k = f_1 +f_2\sqrt{a_1}+f_3\sqrt{a_2}+f_4\sqrt{a_1a_2}$ for some $f_1,f_2,f_3,f_4 \in F^\times$, and so by assumption we get $$g = N_{K/K_3}(k) = (f_1^2-a_1f_2^2-a_2f_3^2+a_1a_2f_4^2) + \sqrt{a_1a_2}(2f_1f_4-2f_2f_3).$$  However since $g \in F^\times$ we must have $f_1f_4 = f_2f_3$.  Our goal is to write $g$ as an element of $N_{K_1/F}(K_1^\times)\cdot N_{K_2/F}(K_2^\times)$, which means we'd like to find $h_1,h_2,h_3,h_4 \in F$ so that $h_1 +h_2 \sqrt{a_1} \in K_1$ and $h_3+h_4\sqrt{a_2} \in K_2$ yield $$g = N_{K_1/F}(h_1+h_2\sqrt{a_1}) \cdot N_{K_2/F}(h_3+h_4\sqrt{a_2}) = (h_1^2-h_2^2a_1)(h_3-h_4^2 a_2).$$  In other words, we need to solve $$f_1^2-a_1f_2^2-a_2f_3^2+a_1a_2f_4^2 = (h_1^2-h_2^2a_1)(h_3-h_4^2 a_2).$$

We proceed by cases.  First, suppose that $f_1=0$.  Hence we must have either $f_2 = 0$ or $f_3 = 0$.  Note if $f_2 = 0$ then our expression for $g$ becomes $$g = -a_2f_3^2+a_1a_2f_4^2 = (f_3^2-f_4^2a_1)(0^2-1^2 a_2).$$ A similar computation settles the case where $f_3 = 0$.  So now suppose that $f_1 \neq 0$, and observe that since $f_4 = \frac{f_2f_3}{f_1}$ we have $$f_1^2-f_2^2 a_1 - f_3^2 a_2+f_4^2 a_1a_2 = \left(f_1^2-f_2^2 a_1\right)\left(1^2-\frac{f_3^2}{f_1^2}a_2\right).$$

With both (\ref{eq:first.fact}) and (\ref{eq:second.fact}) in hand, we can prove the lemma.  If we apply (\ref{eq:second.fact}) to $\frac{f}{a_1^\varepsilon}$ from (\ref{eq:first.fact}), then we see that $$\frac{f}{a_1^\varepsilon} \in N_{K_1/F}(K_1^\times) \cdot N_{K_2/F}(K_2^\times).$$  But since $[a_1]=[1]$, we get the desired result.
\end{proof}

We are now ready for the proof of the main result of this paper.  Our basic strategy is to show that the modules $\hat J$ and $X$ from Theorems \ref{th:trivial.candidate} and \ref{th:construction.of.X} provide the desired decomposition, though in the case where $\dim(\im(T)) = 3$ and $T([N_{K/K_1}(K^\times)] \cap [N_{K/K_2}(K^\times)]) = \{(1,1,1)\}$ we will need to make a small adjustment to $\hat J$ --- removing a single trivial summand --- to achieve our result.

\begin{proof}[Proof of Theorem \ref{th:main.theorem}]
Let $\hat J$ be the module from Theorem \ref{th:trivial.candidate} and let $X$ be the module from Theorem \ref{th:construction.of.X}.  

If we are not in the case where $\dim(\im(T)) = 3$ and $T([N_{K/K_1}(K^\times)] \cap [N_{K/K_2}(K^\times)]) = \{(1,1,1)\}$, then define $\tilde J = \hat J$.  Otherwise, note that in the final case of Theorem \ref{th:construction.of.X} we have a unique $[x_0] \in X^G \cap [F^\times]$, and that $[x_0] \not\in \mathfrak{B}+\mathfrak{C}+\mathfrak{D}$.  Now in the construction of $\hat J$, the summand $Y_F$ is chosen as the span of $\mathcal{F}_0$, where $\mathcal{F}_0$ is an arbitrary basis for a complement of $\mathfrak{B}+\mathfrak{C}+\mathfrak{D}$ within $[F^\times]$ (see the definition of $\mathcal{F}_0$ in Theorem \ref{th:trivial.candidate}).  Since $[x_0] \in [F^\times] \setminus (\mathfrak{B}+\mathfrak{C}+\mathfrak{D})$, we can assume that $[x_0] \in \mathcal{F}_0$.  In this case we define $\tilde Y_F = \sum_{[f] \in \mathcal{F}_0\setminus\{[x_0]\}} \langle [f]\rangle = \bigoplus_{[f] \in \mathcal{F}_0\setminus\{[x_0]\}} \langle [f] \rangle$, and set $$\tilde J = Y_A+Y_V+Y_W+Y_{B}+Y_{C}+Y_{D}+\tilde Y_F = Y_A\oplus Y_V \oplus Y_W \oplus Y_{B}\oplus Y_{C}\oplus Y_{D}\oplus \tilde Y_F.$$  (I.e., the module $\tilde J$ is just the result of removing the summand $\langle [x_0]\rangle$ from $\hat J$.)  

In either case we will show that $J(K) = \tilde J \oplus X$.  Of course we have $\tilde J + X \subseteq J(K)$, and furthermore our construction of $\tilde J$ gives $X^G \cap \tilde J = \{[1]\}$,  so that $\tilde J + X = \tilde J \oplus X$.  Hence we only need to verify that $J(K) \subseteq \tilde J + X$.  We do this by examining the possible isomorphism classes for $\langle [\gamma]\rangle$, where $[\gamma] \in J(K)$.

First, suppose that $\langle [\gamma]\rangle \simeq \mathbb{F}_2$, so that $[\gamma] \in J(K)^G$.  If $[\gamma] \in [F^\times]$, then since $[F^\times] = \hat J^G \subseteq \tilde J^G \oplus X^G$, we have $[\gamma] \in \tilde J+X$.  Otherwise we have $T([\gamma]) \neq (1,1,1)$, in which case by Theorem \ref{th:construction.of.X} there exists some $[x] \in X^G$ with $T([\gamma]) = T([x])$.  We then have $[\gamma][x] \in [F^\times]$, and from the previous case this gives $[\gamma][x] \in \tilde J+X$.  Since $[x] \in \tilde J+X$, we get $[\gamma] \in \tilde J+X$.

Now suppose that $\langle [\gamma] \rangle \simeq \mathbb{F}_2[\overline{G_1}]$.  Lemma \ref{cor:dimension.two.fixed.parts.are.from.F} tells us that $[\gamma]^{1+\sigma_1} \in [F^\times]$, and so $[\gamma]^{1+\sigma_1} \in \mathfrak{B}$.  Corollary \ref{cor:solvable.in.hat.J} tells us that there exists some $[\tilde \gamma] \in \hat J$ so that $[\tilde \gamma]^{1+\sigma_2} = [1]$ and $[\tilde \gamma]^{1+\sigma_1} = [\gamma]^{1+\sigma_1}$; in fact, since $\hat J$ and $\tilde J$ differ by only a trivial summand, and since $[\tilde \gamma] \not\in J(K)^G$, we can assume $[\tilde \gamma] \in \tilde J$ as well.  But then we get $[\gamma][\tilde \gamma] \in J(K)^G$, so by the previous case we have $[\gamma][\tilde\gamma] \in \tilde J+X$.  Since $[\tilde \gamma] \in \tilde J+X$ already, this gives $[\gamma] \in \tilde J+X$.  

The cases where $\langle [\gamma] \rangle$ is isomorphic to either $\mathbb{F}_2[\overline{G_2}]$ or $\mathbb{F}_2[\overline{G_3}]$ follow the same argument as the case $\mathbb{F}_2[\overline{G_1}]$ above.

Now suppose that $\langle[\gamma]\rangle \simeq \Omega^{-1}$, and first consider the case where $T(\langle[\gamma]\rangle^G) = \{(1,1,1)\}$.  By Lemmas \ref{le:supernorm.as.F.detector} and \ref{le:chain.breaking.lemma} we have $[N_{K/K_1}(\gamma)] \in [N_{K/K_1}(K^\times)] \cap[F^\times] = [N_{K_2/F}(K_2^\times)][N_{K_3/F}(K_3^\times)]$, say $[N_{K/K_1}(\gamma)] = [f_{1,2}][f_{1,3}]$, where for $i \in \{2,3\}$ we have $[f_{1,i}] = [N_{K_i/F}(k_i)]$ for some $k_i \in K_i^\times$.  This means that $[f_{1,2}] \in \mathfrak{C}$ and $[f_{1,3}] \in \mathfrak{D}$, so by Corollary \ref{cor:solvable.in.hat.J} there exists $[\gamma_{1,i}] \in \hat J$ with $[\gamma_{1,i}]$ and $[f_{1,i}]$ providing a solution to the appropriate diagram; since $\hat J$ and $\tilde J$ differ only by a trivial summand, we can assume that $[\gamma_{1,i}] \in \tilde J$ for $i \in \{2,3\}$. (See Figure \ref{fig:gamma.is.omega.-1.what.we.have.from.indices.lemma} for a graphical description of these relationships.)
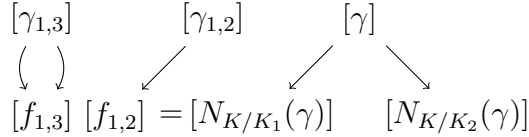
\begin{figure}[!ht]
\begin{tikzpicture}[scale=.65]
\node (G) at (0,0) {$[\gamma]$};
\node (N1) at (-2,-2) {$[N_{K/K_1}(\gamma)]$};
\node at (-3.85,-2) {$=$};
\node (G1D) at (-3,0) {$[\gamma_{1,2}]$};
\node (G1E) at (-6.5,0) {$[\gamma_{1,3}]$};
\node (f1D) at (-5,-2) {$[f_{1,2}]$};
\node (f1E) at (-6.5,-2) {$[f_{1,3}]$};
\node (N2) at (2,-2) {$[N_{K/K_2}(\gamma)]$};

\draw[->] (G) to (N1);
\draw[->] (G1D) to (f1D);
\draw[->, out=240, in=120] (G1E) to (f1E);
\draw[->, out=300, in=60] (G1E) to (f1E);
\draw[->] (G) to (N2);
\end{tikzpicture}
\caption{Decomposing $[N_{K/K_1}(\gamma)]$ in terms of solutions to the diagrams for $\mathfrak{C}$ and $\mathfrak{D}$.}
\label{fig:gamma.is.omega.-1.what.we.have.from.indices.lemma}
\end{figure}

Consider the element $[\tilde \gamma] = [\gamma][\gamma_{1,2}][\gamma_{1,3}]$.  One sees that $[\tilde \gamma]^{1+\sigma_2} = [1]$, and that $[\tilde \gamma]^{1+\sigma_1} = [N_{K/K_2}(\gamma)][f_{1,3}]$.  Hence $\langle [\tilde \gamma] \rangle$ is isomorphic to either $\{[1]\}$ or $\mathbb{F}_2$ or $\mathbb{F}_2[\overline{G_1}]$.  Our previous cases therefore allow us to conclude $[\tilde \gamma] \in \tilde J+X$.  Since $[\gamma_{1,2}],[\gamma_{1,3}] \in \tilde J$, we have $[\gamma] \in \tilde J+X$.

If $T(\langle[\gamma]\rangle^G) \neq \{(1,1,1)\}$ then Lemma \ref{le:indices.when.in.the.image} gives us that precisely one of the following holds
\begin{itemize}
\item $T([N_{K/K_1}(\gamma)]) = (1,a_1,a_1)$ and $T([N_{K/K_2}(\gamma)]) = (1,1,a_1)$; or
\item $T([N_{K/K_1}(\gamma)]) = (1,1,a_1)$ and $T([N_{K/K_2}(\gamma)]) = (a_2,1,a_1)$; or
\item $T([N_{K/K_1}(\gamma)]) = (1,a_1,1)$ and $T([N_{K/K_2}(\gamma)]) = (a_2,1,1)$.
\end{itemize}
In any of these cases our construction for $X$ (see the second case in Theorem \ref{th:construction.of.X}) gives an element $[x] \in [X]$ so that $T([N_{K/K_1}(\gamma)]) = T([N_{K/K_1}(x)])$ and $T([N_{K/K_2}(\gamma)]) =T([N_{K/K_2}(x)])$.  Hence the images of $[\gamma][x]$ under $1+\sigma_1$ and $1+\sigma_2$ both lie in $[F^\times]$, and so $\langle [\gamma][x] \rangle$ is either trivial or falls into one of the previous cases.  We get $[\gamma][x] \in \tilde J+X$, whence $[\gamma] \in \tilde J+X$.

The final case to consider is when $\langle[\gamma]\rangle \simeq \mathbb{F}_2[G]$.  In this case, note that $[N_{K/F}(\gamma)] \in \mathfrak{A}$, and Corollary \ref{cor:solvable.in.hat.J} gives us some element $[\tilde \gamma] \in \hat J$ (which we may assume is in $\tilde J$ since $\tilde J$ and $\hat J$ differ only by a trivial summand) so that $[N_{K/F}(\tilde \gamma)] = [N_{K/F}(\gamma)]$.  From this we get that $\langle [\gamma][\tilde \gamma]\rangle$ is not free, and so is one of the previous isomorphism types.  As usual, this gives us $[\gamma] \in \tilde J+X$.
\end{proof}

\section{Some realizability results}\label{sec:realizability}

Theorem \ref{th:main.theorem} tells us that there are a limited number of summands that could possibly appear in a decomposition of $J(K)$, but is it the case that each of these summand types occurs for at least one biquadratic extension $K/F$?  In this section we offer some partial results concerning this kind of realizability question, focusing particularly on the possible structures for the $X$ summand from Theorem \ref{th:main.theorem}.  For a more complete treatment of this problem of realizing the various summands, the reader is encouraged to consult \cite{CMSSTT} which enhances the current work by exploring its connection to the Brauer group $\text{Br}(F)$.  

The $X$ summand takes on one of $6$ possible structures, with the various possibilities determined by the image of the function $T$ from section \ref{sec:X.candidate} (as detailed in Theorem \ref{th:construction.of.X}).  To determine whether these structures are realizable, we will view the conditions found in Theorem \ref{th:construction.of.X} through the lens of Galois embedding problems via Lemma \ref{le:Galois.groups.for.fixed.elements}.  

First we introduce some terminology.  Note that since $K/F$ is a biquadratic extension with intermediate fields $K_1, K_2$, and $K_3$, if there exists some extension $L/K$ which is Galois over $F$ with $\Gal(L/F) \simeq D_4$, then there is a unique $i \in \{1,2,3\}$ so that $\Gal(L/K_i) \simeq \mathbb{Z}/4\mathbb{Z}$.  We will refer to such an extension as a $D_4$-extension of type $i$. Likewise if there is an extension $L/K$ with $\Gal(L/F) \simeq \mathbb{Z}/4\mathbb{Z} \oplus\mathbb{Z}/2\mathbb{Z}$, then there is a unique $i \in \{1,2,3\}$ so that there exists some field $\widetilde{L}$ with $K_i \subsetneq \widetilde{L} \subsetneq L$ and $\Gal(\widetilde{L}/F) \simeq \mathbb{Z}/4\mathbb{Z}$.  We will refer to such an extension as a $\mathbb{Z}/4\mathbb{Z}\oplus\mathbb{Z}/2\mathbb{Z}$-extension of type $i$.  

Lemma \ref{le:Galois.groups.for.fixed.elements} tells us that if $[\gamma] \in J(K)^G$, then the Galois group of $K(\sqrt{\gamma})/F$ can be computed entirely in terms of $T([\gamma])$.  For example, suppose that $T([\gamma]) = (a_2,1,1)$.  By Lemma \ref{le:Galois.groups.for.fixed.elements} we see that $K(\sqrt{\gamma})/F$ is a $D_4$-extension of type $1$.  Similarly, if $T([\gamma]) = (1,a_1,1)$ or $T([\gamma]) = (1,1,a_1)$, then $K(\sqrt{\gamma})/F$ is a $D_4$-extension of type $2$ or $3$ (respectively).  We also have that $T([\gamma]) \in  \{(a_2,a_1,1),(a_2,1,a_1),(1,a_1,a_1)\}$ implies that $K(\sqrt{\gamma})/F$ is a $\mathbb{Z}/4\mathbb{Z}\oplus \mathbb{Z}/2\mathbb{Z}$-extension (of types $3$,$2$, and $1$, respectively).  If $T([\gamma]) = (a_2,a_1,a_1)$, then $K(\sqrt{\gamma})/F$ is a $Q_8$-extension.  Finally, if $T([\gamma]) = (1,1,1)$ then $\Gal(K(\sqrt{\gamma})/F)$ is elementary $2$-abelian of rank $3$.  Since each of the possible values of $T([\gamma])$ yields a distinct Galois group, this dictionary works both ways: the structure of the Galois group of a given $K(\sqrt{\gamma})/F$  determines the value of $T([\gamma])$.

Happily, these types of embedding problems have already been studied extensively.  For example, in \cite{Kiming} one finds that a quadratric extension $E(\sqrt{a})/E$ embeds in a $\mathbb{Z}_4$-extension if and only if $a=x^2+y^2$ for $x,y \in E$.  Likewise, a biquadratic extension $E(\sqrt{a},\sqrt{b})/E$ embeds in a $D_4$-extension $L/E$ for which $\Gal(L/E(\sqrt{b})) \simeq \mathbb{Z}/4\mathbb{Z}$ if and only if $b = ay^2-x^2$ for some $x,y \in E$.  Finally, a biquadratic extension $E(\sqrt{a},\sqrt{b})/E$ embeds in a $Q_8$-extension if and only if there are $e_1,e_2,e_3,f_1,f_2,f_3 \in E$ with $a = \sum_{i=1}^3 e_i^2$ and $b = \sum_{i=1}^3 f_i^2$ and $\sum_{i=1}^3 e_if_i=0$.  Hence we can determine if a given biquadratic extension $K/F$ has elements $[\gamma] \in J(K)^G$ with prescribed values under $T$ by determining whether certain equations hold over $F$.

\begin{example}
Let $F = \mathbb{Q}$ and $K = \mathbb{Q}(\sqrt{7},\sqrt{-5})$.  (Following our previous conventions, this means $a_1 = 7$ and $a_2 = -5$.)  None of the elements from $\{7,-5,-35\}$ be written as a sum of rational squares, and hence $K/F$ does not embed in any type of $\mathbb{Z}/4\mathbb{Z}\oplus\mathbb{Z}/2\mathbb{Z}$-extension.  Hence no element from $\{(a_2,a_1,1),(a_2,1,a_1),(1,a_1,a_1)\}$ is in $\text{im}(T)$.  We also can clearly see that $7=-5y^2-x^2$ has no rational solutions, so $K/F$ does not embed in a $D_4$-extension of type $1$; therefore $(a_2,1,1) \not\in \text{im}(T)$.  Likewise $-5 = -35y^2-x^2$ and $-35 = -5y^2-x^2$  have no rational solutions.  For example, a rational solution to $-5=-35y^2-x^2$ would imply an integral solution to $u^2=7v^2+5w^2$ for which $5 \nmid u$ and $5\nmid v$.  One sees this is impossible by examining this equation modulo $5$.  Because these equations have no rational solutions, it follows that $K/F$ does not embed in a $D_4$-extension of type $2$ or $3$ either.  Hence $\{(1,a_1,1),(1,1,a_1)\} \not\in \text{im}(T)$.  Finally, since $-5$ is conspicuously not a sum of three rational squares, we have that $K/F$ does not embed in a $Q_8$-extension, and so $(a_2,a_1,a_1) \not\in \text{im}(T)$.  Hence $\text{im}(T) = \{(1,1,1)\}$, and by Theorem \ref{th:construction.of.X} we have $X = \{[1]\}$.
\end{example}

\begin{example}
Let $F = \mathbb{Q}$ and $K = \mathbb{Q}(\sqrt{7},\sqrt{-1})$.  We see that $K/F$ does not embed in any $\mathbb{Z}/4\mathbb{Z}\oplus \mathbb{Z}/2\mathbb{Z}$-extension since none of $7,-1$, nor $-7$ is a sum of two rational squares; it does not embed in a $D_4$-extension of type $1$ or $3$ since $7 = -y^2-x^2$ and $-7 = -y^2-x^2$ have no rational solutions; and it does not embed in a $Q_8$-extension since $-1$ is not a sum of three rational squares.  It does, however, embed in a $D_4$-extension of type $2$ since $-1 = -7y^2-x^2$ has a rational solution.  Hence $\text{im}(T) = \{(1,1,1),(1,a_1,1)\}$, and so $X \simeq \mathbb{F}_2$.
\end{example}

\begin{example}
Let $F = \mathbb{Q}$ and $K = \mathbb{Q}(\sqrt{2},\sqrt{-1})$.  Since $2$ is a sum of two rational squares but $-1$ and $-2$ are not, we see that $K/F$ embeds in a $\mathbb{Z}/4\mathbb{Z}\oplus\mathbb{Z}/2\mathbb{Z}$-extension of type $1$, but not of types $2$ or $3$.  It's also the case that $2 = -y^2-x^2$ has no rational solutions, but $-1 = 2y^2-x^2$ and $-2 = 2y^2-x^2$ do have rational solutions, and hence $K/F$ embeds in $D_4$-extensions of types $2$ and $3$, but not type $1$.  We also have that $-1$ is not a sum of three rational cubes, so $K/F$ does not embed in a $Q_8$-extension.    Taken together, this means that $\text{im}(T) = \{(1,1,1),(1,a_1,a_1),(1,a_1,1),(1,1,a_1)\}$, which is one of the coordinate planes (the ``$yz$-plane").  Hence from Theorem \ref{th:construction.of.X}, we have $X \simeq \Omega^{-1}$.
\end{example}

\begin{example}
Let $F = \mathbb{Q}$ and $K = \mathbb{Q}(\sqrt{5},\sqrt{13})$.  We know that each of $5$, $13$, and $65$ can be written as a sum of two rational (indeed, integral) squares, and hence $K/F$ embeds in $\mathbb{Z}/4\mathbb{Z}\oplus\mathbb{Z}/2\mathbb{Z}$-extensions of types $1$, $2$, and $3$.  Therefore $\{(1,1,1),(a_2,a_1,1),(a_2,1,a_1),(1,a_1,a_1)\} \subseteq \text{im}(T)$.  On the other hand, there is no rational solution to $5 = 13y^2-x^2$, since such a solution would imply an integral solution to $5u^2 = 13v^2-w^2$.  (After ensuring that $5$ does not divide all of $u$,$v$, and $w$, one examines the equation modulo $5$.) Hence $K/F$ does not embed in a $D_4$-extension of type $1$, and so $(a_2,1,1) \not\in \text{im}(T)$.  Since $T$ is an $\mathbb{F}_2$-space, we get $\{(1,1,1),(a_2,a_1,1),(a_2,1,a_1),(1,a_1,a_1)\} = \text{im}(T)$.  By Theorem \ref{th:construction.of.X}, we have $X \simeq \mathbb{F}_2\oplus\mathbb{F}_2$.  
\end{example}

\begin{example}
Let $F = \mathbb{Q}$ and $K = \mathbb{Q}(\sqrt{5},\sqrt{41})$.  Since $5$, $41$, and $205$ are all expressible as sums of two rational squares, and since we can write $5 = (2)^2+(1)^2+0^2$ and $41 = (-1)^2+(2)^2+(6)^2$, we see that $\{(a_2,a_1,1),(a_2,1,a_1),(1,a_1,a_1),(a_2,a_1,a_1)\} \subseteq \text{im}(T)$.  Hence $\dim(\text{im}(T)) = 3$ in this case, and we have either $X \simeq \Omega^{-1}\oplus\Omega^{-1}$ or $X \simeq \Omega^{-2}$ (depending on whether $[N_{K/K_1}(K^\times)] \cap [N_{K/K_2}(K^\times)] \subseteq [\mathbb{Q}^\times])$.  
\end{example}

Readers who are familiar with Galois embedding problems will recognize that absence of a key player from our discussion above: the Brauer group.  Indeed, the solvability of each of the embedding problems we've discussed is encoded in the vanishing of certain element(s) drawn from $\langle (a_1,a_1),(a_1,a_2),(a_2,a_2)\rangle \subseteq \text{Br}(\mathbb{Q})$.  The relationship between embedding problems and Galois cohomology has been studied extensively; for a small sampling, see \cite{JY,Kiming,Michailov4,Michailov5}.  The focus of the follow-up paper \cite{CMSSTT} is to reinterpret the decomposition of $J(K)$ provided by Theorem \ref{th:main.theorem} through the lens of certain equations in $\text{Br}(F)$.  In particular, this will allow us to compute the multiplicities of the various summands by analyzing subspaces within $\text{Br}(F)$, and ultimately show that all listed summand types from Theorem \ref{th:main.theorem} are realizable.


\end{document}